\documentclass[reqno]{amsart}

\usepackage{enumerate,amsmath,amssymb}
\usepackage{graphicx}
\usepackage{lpic}
\usepackage{hyperref}

\theoremstyle{plain}
 
\newtheorem{thm}{Theorem}[section] 
 
\newtheorem{cor}[thm]{Corollary}
\newtheorem{corollary}[thm]{Corollary} 
\newtheorem{lem}[thm]{Lemma} 
\newtheorem{lemma}[thm]{Lemma}
\newtheorem{prop}[thm]{Proposition}

\newtheorem{conjecture}[thm]{Conjecture}

\theoremstyle{definition} 
\newtheorem{defn}[thm]{Definition}

\renewcommand{\epsilon}{\varepsilon}
\let\phi\varphi

\let\<\langle
\let\>\rangle

\DeclareMathAlphabet{\doba}{U}{msb}{m}{n}

\gdef\mR{\doba{R}}

\gdef\mZ{\doba{Z}}

\def\grad{{\mathop{\rm grad}}}
\def\scal{{\mathop{\rm scal}}}

\def\Spin{{\mathop{\rm Spin}}}
\def\SO{{\mathop{\rm SO}}}
\def\supp{{\mathop{\rm supp}}}
\def\vol{{\mathop{\rm vol}}}

\let\spec\Spec

\newcommand{\definedas}{\mathrel{\raise.095ex\hbox{\rm :}\mkern-5.2mu=}}

\def\Rinv{{\mathcal R}^{\rm inv}} 
\def\Rpsc{{\mathcal R}^{\rm psc}} 

\def\Ospin{\Omega^{\rm spin}} 
\def\Oinv{\Omega^{\rm inv}}

\def\rel{{ \, \rm{rel} \, }} 
\def\conc{{\widetilde{\pi}_0}} 

\begin{document}


\title 
[Invertible Dirac operators and handle attachments]
{Invertible Dirac operators and handle attachments on manifolds with
 boundary}

\author{Mattias Dahl} 
\address{Mattias Dahl \\ 
Institutionen f\"or Matematik \\
Kungliga Tekniska H\"ogskolan \\
100 44 Stockholm \\
Sweden}
\email{dahl@math.kth.se}

\author{Nadine Gro\ss e} 
\address{Nadine Gro\ss e \\ 
Mathematisches Institut \\ 
Universt\"at Leipzig \\
04009 Leipzig \\
Germany}
\email{grosse@math.uni-leipzig.de}




\begin{abstract}
For spin manifolds with boundary we consider Riemannian metrics which 
are product near the boundary and are such that the corresponding Dirac 
operator is invertible when half-infinite cylinders are attached at the 
boundary. The main result of this paper is that these properties of a 
metric can be preserved when the metric is extended over a handle of 
codimension at least two attached at the boundary. Applications of this 
result include the construction of non-isotopic metrics with invertible 
Dirac operator, and a concordance existence and classification theorem.
\end{abstract}

\maketitle

\tableofcontents

\section{Introduction}

Surgery constructions can be used in differential geometry to show 
that under certain assumptions a geometric structure on a manifold 
can be transported to another manifold by applying certain surgeries.
The most prominent examples are probably the results concerning 
positive scalar curvature, for an overview see \cite{rosenberg_stolz_01}. 
In \cite[Theorem~A]{gromov_lawson_80} Gromov and Lawson showed that if 
a closed manifold $M$ possesses a metric of positive scalar curvature,
then a manifold obtained from $M$ by a surgery of codimension $\geq 3$
also possesses a metric with positive scalar curvature. This was
generalized to handle attachments on manifolds with boundaries by 
Gajer in \cite{gajer_87} and Carr in \cite{carr_88}. 
Those results not only gave existence results for positive scalar
curvature but also allowed conclusions on the topology of the space of
all metrics of positive scalar curvature
$\Rpsc(M) \subset \mathcal{R}(M)$. For example in \cite{carr_88},
\cite[Chapter~4, Theorem~7.7]{lawson_michelsohn_89}, it is shown that 
that $\Rpsc(M)$ has infinitely many connected components if 
$\dim M = 4m - 1$, $m \geq 2$. This result is sharp in the sense that 
$\Rpsc(S^3)$ is connected, see \cite{marques_12}. There are also newer 
works that examine the higher homotopy groups of the moduli space of 
positive scalar curvature, see for example \cite{chernysh_04}, 
\cite{botvinnik_hanke_schick_walsh_10}.

If the manifold is spin, similar results can be obtained for metrics 
with invertible Dirac operator.

\begin{thm} \cite[Theorem~1.2]{ammann_dahl_humbert_09}
\label{am_da_hu_09} 
Let $(M,g)$ be a closed Riemannian spin manifold of dimension $n$ and 
let $\widetilde{M}$ be obtained from $M$ by surgery of dimension $k$
where $0 \leq k \leq n-2$. Then $\widetilde{M}$ carries a metric
$\tilde{g}$ for which $\dim \ker D^{\tilde{g}} \leq \dim \ker D^g$. 
\end{thm}

If the codimension is $\geq 3$ this result is a special case of 
\cite[Theorem~1.2]{baer_dahl_02}. In the spirit of the generalization
of the surgery result for positive scalar curvature to manifolds with
boundary, the first author showed in \cite[Proposition~2.5]{dahl_08} 
that a metric with invertible Dirac operator on a closed spin manifold 
$M$ can be extended to a metric with invertible Dirac operator over 
the trace $W$ of a surgery of codimension $\geq 3$ on $M$.

From the Lichnerowicz formula it follows that every metric with 
positive scalar curvature on a closed spin manifold has invertible 
Dirac operator, that is $\Rpsc(M) \subset \Rinv(M) \subset 
\mathcal{R}(M)$. Thus, obstructions to invertibility of the Dirac 
operator give obstructions to the existence of positive scalar 
curvature. One of the main obstruction is given by the index of the 
Dirac operator which is equal to the $\alpha$-genus of the manifold, 
compare \cite{hitchin_74}.

The aim of this paper is to generalize Theorem~\ref{am_da_hu_09} to
handle attachments for manifolds with boundary. This will also be an
extension of \cite[Proposition~2.5]{dahl_08} to codimension $2$. For
that we have to fix a notion of invertibility for Dirac operators on
complete but non-compact manifolds. We will use the following
conventions. 

On a compact manifold with boundary we always assume that any
Riemannian metric has a product structure near the boundary. When
considering spectral properties of the Dirac operator we attach
half-infinite cylinders (with the natural product metric) at the
boundary and consider the Dirac operator acting on smooth
$L^2$-sections on the resulting complete manifold. 

Our goal is to show that if a metric $g$ with invertible Dirac
operator is given on a manifold with boundary $M$, then there is 
a metric $g''$ with invertible Dirac operator on the manifold $M''$ 
obtained by attaching a handle of codimension $\geq 2$ at the boundary 
of $M$ (see below). Moreover the metrics $g$ and $g''$ coincide outside 
an arbitrarily small neighbourhood of the handle attachment sphere.

For a compact spin manifold $M$ we denote by $\Rinv(M)$ the set of
metrics with invertible Dirac operator and product structure near the
boundary. 

We assume as given a compact Riemannian spin manifold with boundary
$(M,g)$, where $\dim M = n + 1$, and we assume that $D^g$ is
invertible. Let $\partial g \definedas g|_{\partial M}$ be the induced
metric on the boundary so that $g = \partial g + dt^2$ in a
neighbourhood of the boundary. The handle attachment (and the
corresponding surgery on the boundary) is specified by a 
spin-structure preserving embedding 
\[
f : S^k \times B^{n-k} \to \partial M .
\]
For $k \geq 2$ the spin structure on $S^k \times B^{n-k}$ is uniquely 
determined. But for $k=1$ there are two spin structures where we only 
allow the spin structure on $S^1$ which bounds a disks. Otherwise we 
cannot assure that the handle attachment will be compatible with the 
spin structures. From now on, when having a handle attachment for $k=1$ 
we always mean implicitly that $S^1$ is equipped with the bounding spin 
structure.

The image of $S^k \times \{ 0 \}$ is called the handle attachment
sphere and is denoted by $S$. In Figure~\ref{figure_intro_a} the
surgery sphere (a zero dimensional $S^0$) is represented by the
two dots in the boundary of $M$. 

\begin{figure}[h]
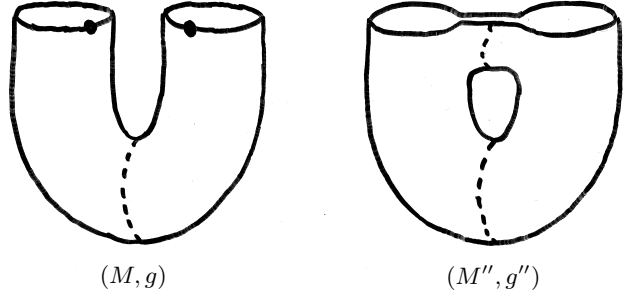

\begin{lpic}[]{FIGS/INTRO_A(0.33)}
\lbl[]{75,10;{\small $(M, g)$}}
\lbl[]{220,10;{\small$(M'', g'')$}}
\end{lpic}
\caption{Handle attachment.}
\label{figure_intro_a}
\end{figure}

More precisely (for pictures see the strategy described in
Section~\ref{strategy}), the manifold $M''_\infty$ will be obtained
from $M_\infty$ in the following way: $S^k \times B_{-}^{n-k+1}$ will
be embedded in $M_\infty$, where $B_{-}^{n-k+1}$ is the
$(n-k+1)$-dimensional half disk 
\[
B_{-}^{n-k+1} \definedas
\left\{ 
(x_1,\ldots, x_{n-k+1}) \mid \sum x_i^2 \leq 1, x_{n-k+1} \leq 0 
\right\}.
\]
The image of this embedding will be removed and replaced by
$B^{k+1} \times S_{-}^{n-k-1}$, where
$S_{-}^{n-k-1} \definedas S^{n-k-1} \cap B_{-}^{n-k} \cong B^{n-k-1}$
denotes the lower hemisphere. Restricting the handle attachment to the 
boundary $\partial M$ the embedded $S^k \times B^{n-k}$ is replaced by
$B^{k+1}\times S^{n-k-1}$. Thus, on the boundary we also have a
$k$-dimensional surgery.

\begin{thm} \label{thm_handle_attachment}
Let $M$ be a manifold with boundary and $g \in \Rinv(M)$. Let
$M''$ be obtained by a handle attachment as described above
where $n-k \geq 2$. Then for any given neighbourhood of the surgery
sphere there is a metric $g'' \in \Rinv(M'')$ such that $g'' = g$
outside that neighbourhood.
\end{thm}

Note, that the construction described above is exactly the
attachment of a $k$-handle at the boundary $\partial M$.

We indicate some applications of this result.

In Section~\ref{section_non-isotopic} we will show that the space 
$\Rinv(S^3)$ is---in contrast to $\Rpsc(S^3)$---not path-connected. 
More generally, let $(M, g\in \Rinv (M))$ be a closed $3$-dimensional 
Riemannian spin manifold. Then, in 
Proposition~\ref{non_concordant_metrics} we construct infinitely many 
non-concordant metrics with invertible Dirac operators that are pairwise 
bordant but not concordant. This generalizes a result from the first 
author \cite[Theorem~3.3]{dahl_08} to dimension $3$.

In Section~\ref{section_concordance} we discuss the concordance 
classification of metrics with invertible Dirac operator. These 
considerations are mainly based on the work of Stolz \cite{stolz_xx}. 
The case of invertible Dirac operators is easier than the positive 
scalar curvature case since we have larger range of handle attachments 
available, and we do not have to take the fundamental group into 
account, see Section \ref{Rngroups}. Hence, one motivation to examine 
the invertible Dirac operator case is that it gives a simplified 
illustration of this circle of ideas.

In Section~\ref{section_genericity} we will prove a genericity 
result. Proposition~\ref{open_dense} gives that if the subset 
$\Rinv(M \rel h)\subset \mathcal{R} (M \rel h)$ is 
non-empty, then it is open with respect to the $C^1$-topology and 
dense with respect to the $C^\infty$-topology. Here the notation 
$\rel h$ means that we only consider metrics with a fixed boundary 
metric $h \in \Rinv (\partial M)$. Thus, let a manifold with boundary 
$(M,g_M)$ be obtained by handle attachments of codimension $\geq 2$ 
from a manifold with boundary $(N, g_N\in \Rinv(N))$. Then the subset 
$\Rinv(M \rel g_M|_{\partial M})$ is generic (in the sense above) 
in $\mathcal{R} (M \rel g_M|_{\partial M})$, 
see Theorem~\ref{thm_generic_metrics}.

\subsection*{Acknowledgements}

The authors want to express their deep gratitude to B. Ammann, 
O. Andersson, and E. Humbert who have been much involved at earlier 
stages in this project and have contributed many helpful comments. 

\section{Preliminaries}

The following subsections contain material in preparation for the
proof of Theorem~\ref{thm_handle_attachment}.

\subsection{Notation}

The flat metric on $\mR^n$ is denoted by $\xi^n$, the round metric of
radius $1$ on $S^n$ is denoted by $\sigma^n$. Let $B^n(r)$ denote the 
$n$-dimensional ball of radius $r$, and let $B^n \definedas B^n(1)$. 
Let $S^n(r)$ denote the $n$-dimensional sphere of radius $r$, and let
$S^n \definedas S^n(1)$. 

The spinor bundle of a Riemannian spin manifold $(M,g)$ is denoted by
$\Sigma M$ (The construction of the spinor bundle and its dependence
on the Riemannian metric is discussed in
Section~\ref{compare_section}). The spinor bundle is a complex vector
bundle which is fiberwise equipped with a hermitian metric denoted by
$\< \cdot, \cdot \>_g$. Sections of the spinor bundle are called 
spinors. The space of all smooth spinors with compact support, denoted 
by $C_c^\infty(\Sigma M)$, carries a scalar product
\[
(\phi_1, \phi_2)_g
= \int_M \< \phi_1, \phi_2 \>_g \, dv^g ,
\quad \phi_1,\phi_2\in C_c^\infty(\Sigma M).
\]
The completion of $C_c^\infty(\Sigma M)$ with respect to this scalar
product is denoted by $L^2(\Sigma M, g)$. Let $D^g$ be the corresponding
classical Dirac operator. We denote by $H^1(\Sigma M, g)$ those
spinors $\phi \in L^2(\Sigma M, g)$ with 
$\Vert \nabla^g \phi \Vert_{L^2(\Sigma M, g)} + 
\Vert \phi \Vert_{L^2(\Sigma M, g)} < \infty$, where $\nabla^g$ is the
Levi-Civita connection lifted to the spinor bundle. 

In case the Riemannian metric, the underlying manifold or both are 
clear from the context we will abbreviate to $L^2(\Sigma M)$, $L^2(g)$ 
or simply $L^2$. The same will be done for $H^1(\Sigma M, g)$.

\subsection{Spectral theory for manifolds with cylindrical ends}

Before specializing to manifolds with cylindrical ends we review some
facts on the spectrum on complete manifolds as can be found for example
in \cite{baer_05}.

Let $(M,g)$ be a complete Riemannian spin manifold with Dirac operator
$D^g$. Let $\lambda \in \mathbb{C}$ be in the spectrum of $D^g$.
Then $\lambda$ is real and at least one of the following holds:
\begin{itemize}
\item 
$\lambda$ is an eigenvalue, which means that there is a nonzero spinor
$\phi \in C^\infty(\Sigma M) \cap L^2(\Sigma M)$ with
$D^g\phi = \lambda \phi$. Such a spinor $\phi$ is called an eigenspinor 
corresponding to $\lambda$. The set of all eigenspinors to $\lambda$ is 
the eigenspace of $\lambda$. Eigenspinors to different eigenvalues are 
orthogonal in $L^2(\Sigma M)$.
\item 
$\lambda$ is in the essential spectrum, this means that there is a
sequence $\phi_j$ of compactly supported smooth spinors which are
orthonormal in $L^2(\Sigma M, g)$ and satisfy 
$\Vert ( D - \lambda)\phi_j \Vert_{L^2(\Sigma M, g)} \to 0$ as 
$j \to \infty$. 
\end{itemize}
In particular, $\lambda$ could be both an eigenvalue and an element in
the essential spectrum, for example if it is an eigenvalue for which the 
corresponding eigenspace is infinite-dimensional.

Moreover, the following decomposition principle holds for the
essential spectrum. Let $M_1$ and $M_2$ be two complete Riemannian
spin manifolds and $K_i \subset M_i$ be compact subsets such that
$M_1 \setminus K_1$ and $M_2 \setminus K_2$ are isometric. Then the
essential spectrum of the Dirac operators on $M_1$ and $M_2$ are the
same.

If $M$ is closed, then the spectrum consists only of eigenvalues with
finite multiplicities.

Next we state properties of the spectrum on manifolds with
cylindrical ends. To be more precise, the manifold $M$ is assumed to
have a neighbourhood $\partial M \times (-t_0, 0]$ of its boundary
where the metric has the product form $\partial g + dt^2$. Let 
$M_\infty$ be the manifold $M$ with half-infinite cylindrical ends
attached,
\[ 
(M_\infty,g) 
\definedas 
(M,g) \cup_{\partial M} 
(\partial M \times [0,\infty), \partial g + dt^2 ).
\]
By a slight abuse of notation we use the same symbol $g$ for the 
metric on $M$ and on $M_\infty$. Note that $(M_\infty,g)$ contains a
cylindrical part 
$(\partial M \times (-t_0,\infty), \partial g + dt^2 )$.

The Dirac operator $D^g$ has a self-adjoint extension to a bounded
operator 
\[
D^g : H^1(\Sigma M_{\infty}) \to L^2(\Sigma M_{\infty}) ,
\]
see \cite[Section~3.6.2]{bleecker_booss-bavnbek_04}. This operator is
invertible with a bounded inverse if and only if it has a spectral gap
$(-\lambda,\lambda)$ around zero, that is if there is $\lambda > 0$
such that $\Vert D^g \phi \Vert_{L^2} \geq 
\lambda \Vert \phi \Vert_{L^2}$ for all 
$\phi \in H^1(\Sigma M_{\infty})$. We call the spectral gap maximal if 
$\lambda$ or $-\lambda$ is in the spectrum of $D^g$.

The norms $\Vert \phi \Vert_{H^1(g)}$ and 
$\Vert \phi\Vert_{L^2(g)}+\Vert D^g\phi\Vert_{L^2(g)}$ are
equivalent on cylindrical manifolds which follows from the
corresponding statement on compact manifolds, see for example
\cite[Corollary~3.2.4]{ammann_habil}. 

The boundary manifold $(\partial M, \partial g)$ is a compact manifold 
without boundary which carries an induced spin structure 
\cite[Page~200]{milnor_63}. Thus, the Dirac operator $D^{\partial g}$
is a self-adjoint operator 
$H^1(\Sigma \partial M) \to L^2(\Sigma \partial M)$ with discrete
spectrum. 

\begin{prop} \label{propbdrygap}
Assume that $D^{\partial g}$ has a maximal spectral gap
$(-\Lambda, \Lambda)$, $\Lambda > 0$, around zero. That is,
assume 
$ \Vert D^{\partial g} \phi \Vert_{L^2} \geq \Lambda \Vert \phi \Vert_{L^2}$ 
for all $\phi \in H^1(\Sigma \partial M)$. Then
\begin{enumerate}
\item 
\cite[Lemma~3.20]{bleecker_booss-bavnbek_04}
\cite[Section~4]{mueller_94}
in the interval $(-\Lambda, \Lambda)$ the spectrum of 
$D^g : H^1(\Sigma M_{\infty}) \to L^2(\Sigma M_{\infty})$ consists of
finitely many eigenvalues of finite multiplicity. Further, the
essential spectrum of $D^g$ is equal to 
$(-\infty,-\Lambda] \cup [\Lambda,\infty)$.
\item 
\cite[Lemma~3.21]{bleecker_booss-bavnbek_04}
\cite[Section~4]{mueller_94}
any eigenspinor of $D^g$ on $M_{\infty}$ to the eigenvalue $\lambda$
decays exponentially with 
rate $\sqrt{\Lambda^2-\lambda^2}$ on the cylindrical end,
\item 
In particular, for the exponential decay of an eigenspinor
$\phi$ corresponding to an eigenvalue $\lambda$ with 
$\Lambda^2 \geq \lambda^2$ it holds that 
\[ 
\int_{\partial M\times \{ t_1\}} |\phi|^2 \, dv^{\partial g} 
\leq 
2^{-1} (\Lambda^2 - \lambda^2)^{-1/2} 
e^{-2\sqrt{\Lambda^2-\lambda^2}(t_1 - 1)}
\int_{\partial M\times [0,1]} |\phi|^2 \, dv^g
\]
for $0 \leq t_1$.
\end{enumerate}
\end{prop}

We need \textit{(3)} as a quantitative version of \textit{(2)}. 
It shows that the decay rate only depends on $\lambda, \Lambda$ 
but not on $g$.

\begin{proof}[Proof of Part (3)] 
We differentiate 
$l(t)^2 \definedas \int_{\partial M_t} |\phi|^2 \, dv^{\partial g}$ 
where $\partial M_t \definedas \partial M\times \{t\}$ and we obtain 
$l'(t) l(t) = \int_{\partial M_t} 
\< \nabla^g_{\partial_t} \phi,\phi \> \, dv^{\partial g}$. Here and in the 
rest of the proof, $\<.,.\>$ denotes the real part of the hermitian 
scalar product. Differentiating again and using the Cauchy-Schwarz 
inequality, we get
\[ \begin{split}
l''(t)l(t) + l'(t)^2 
&= 
\int_{\partial M_t} 
\< \nabla^g_{\partial_t}\phi,\nabla^g_{\partial_t} \phi \> 
\, dv^{\partial g} 
+ \int_{\partial M_t} 
\< ( \nabla^g_{\partial_t} )^2 \phi, \phi \> 
\, dv^{\partial g} \\ 
&\geq 
\frac{
\left( \int_{\partial M_t} \< \nabla^g_{\partial_t} \phi, \phi \> 
\,dv^{\partial g} \right)^2} 
{\int_{\partial M_t} | \phi |^2 \, dv^{\partial g}}
+ 
\int_{\partial M_t} 
\< ( \nabla^g_{\partial_t} )^2 \phi, \phi \> 
\, dv^{\partial g}, 
\end{split} \]
and, thus,
\[
l''(t) l(t) \geq 
\int_{\partial M_t} \< ( \nabla^g_{\partial_t} )^2 \phi, \phi \> 
\, dv^{\partial g}.
\]
Using the Schr\"odinger-Lichnerowicz formula and 
$\scal^g = \scal^{\partial g}$ we write the square of the Dirac
operator $D^g$ on the cylinder as 
\[
(D^g)^2 
= ( D^{\partial g} )^2 + (\nabla^g_{\partial_t})^*\nabla^g_{\partial_t}
= ( D^{\partial g} )^2 - (\nabla^g_{\partial_t})^2 .
\] 
We obtain 
\[ \begin{split}
\lambda^2 l(t)^2
&= 
\int_{\partial M_t} \< (D^g)^2 \phi, \phi \> \, dv^{\partial g} \\ 
&=
\int_{\partial M_t} |D^{\partial g}\phi|^2 \, dv^{\partial g} 
-
\int_{\partial M_t} 
\< (\nabla^g_{\partial_t})^2 \phi, \phi \> 
\, dv^{\partial g} \\ 
&\geq 
\Lambda^2 l(t)^2 - l''(t)l(t). 
\end{split} \]
Thus, $l''(t) \geq (\Lambda^2 - \lambda^2) l(t)$. Note that $l(t)>0$ 
for all $t$, since the zero set of an eigenspinor has zero $n-1$ 
Hausdorff measure, see \cite{baer_99}. Next we will show
that $l' \leq 0$. Since $\int_0^\infty l(t)^2 \, dt < \infty$ we have 
\[
\lim_{t \to \infty} 
\int_{\partial M_t} \< \nabla^g_{\partial_t} \phi, \phi \> \, dv^{\partial g}
=
\lim_{t \to \infty} l'(t)l(t)
= 
\lim_{t\to\infty} \left( \frac{1}{2} l(t)^2\right)'
= 0, 
\]
which is used in the third step of the following computation.
\[ \begin{split}
\lambda^2 \int_T^\infty l(t)^2 \, dt
&= 
\int_T^\infty \int_{\partial M_t} \< (D^g)^2 \phi, \phi \> 
\, dv^{\partial g} dt \\ 
&= 
\int_T^\infty \int_{\partial M_t} 
\Big(
\< (D^{\partial g})^2 \phi, \phi \> 
- 
\<(\nabla^g_{\partial_t})^2 \phi,\phi\> 
\Big)
\, dv^{\partial g} dt \\ 
&= 
\int_T^\infty \int_{\partial M_t} 
\left( |D^{\partial g} \phi|^2 + |\nabla^g_{\partial_t} \phi|^2 \right)
\, dv^{\partial g} dt
+ 
\int_{\partial M_T} \< \nabla^g_{\partial_t}\phi, \phi \> 
\, dv^{\partial g} , \\ 
&\geq
\Lambda^2 \int_T^\infty l(t)^2 \, dt + l'(T)l(T).
\end{split} \]
Since $\Lambda^2 \geq \lambda^2$ we conclude that $l'(T) \leq 0$. 
Together with $l''(t) \geq (\Lambda^2 - \lambda^2) l(t)$ we have
\[ 
l(t) \geq e^{\sqrt{\Lambda^2-\lambda^2}(t_1-t)} l(t_1)
\]
for $t \geq 0$. Integrating and applying the Cauchy-Schwarz
inequality we get 
\[ \begin{split}
l(t_1)
&\leq 
\int_0^1 l(t) e^{-\sqrt{\Lambda^2-\lambda^2} (t_1-t)} \, dt \\
&\leq
\left( \int_0^1 l(t)^2 \, dt \right)^{1/2}
\left( \int_0^1 e^{-2\sqrt{\Lambda^2 - \lambda^2} (t_1 - t)} 
\, dt \right)^{1/2} \\ 
&\leq
{2}^{-1/2}(\Lambda^2 - \lambda^2)^{-1/4}
e^{-\sqrt{\Lambda^2-\lambda^2}(t_1-1)}
\left( \int_0^1 l(t)^2 \, dt \right)^{\frac{1}{2}}.
\end{split} \]
\end{proof}

\begin{lemma} \label{C^2-estimate}
Let $(M,g)$ be a Riemannian manifold, let $K \subset M$ be a compact
subset, and let $\Lambda > 0$. Then there is a constant 
$C = C(K,M,g,\Lambda)$ such that
\[ 
\Vert \phi \Vert_{C^2(K,g)} \leq C \Vert \phi \Vert_{L^2(M,g)}
\]
for any spinor $\phi$ on $(M,g)$ satisfying $D^g \phi = \lambda \phi$
where $|\lambda| < \Lambda$.
\end{lemma}

Note that $M$ is not assumed to be compact. The proof of 
Lemma~\ref{C^2-estimate} is similar to Lemma~2.2 in 
\cite{ammann_dahl_humbert_09}.

\begin{lemma}[Ascoli's Theorem, {\cite[Theorem~1.34]{adams_fournier_03}}] 
\label{lemma_Ascoli}
Let $(M,g)$ be a Riemannian manifold and let $K \subset M$ be a
compact subset. Suppose that $\phi_i$ is a bounded sequence in
$C^2(K)$, then a subsequence of $\phi_i$ converges in $C^1(K)$.
\end{lemma}

\subsection{Comparing spinors for different metrics}
\label{compare_section}

Let $M$ be an $n$-dimensional spin manifold with Riemannian metrics
$g$ and $g'$. In this subsection we review the method for comparing 
spinors for $g$ and $g'$ following Bourguignon
and Gauduchon \cite{bourguignon_gauduchon_92}.

There is a unique endomorphism $b_{g'}^g$ of $TM$ which is positive,
symmetric with respect to $g$ and satisfies 
$g(X,Y) = g^\prime(b_{g'}^gX, b_{g'}^gY)$ for all $X,Y \in TM$. Since
$b_{g'}^g$ maps $g$-orthonormal frames to $g'$-orthonormal frames,
this gives an $\SO(n)$-principal bundle map 
$b_{g'}^g: \SO(M,g)\to \SO(M,g')$. If the spin structures $\Spin(M,g)$
and $\Spin(M, g')$ are equivalent then the map $b_{g'}^g$ lifts to a
$\Spin(n)$-principal bundle map 
$\beta_{g'}^g: \Spin(M,g)\to \Spin(M, g')$. From this we get a map 
between the spinor bundles $\Sigma^g M$ and $\Sigma^{g'} M$ which we
will denote with the same symbol,
\begin{align*}
\beta_{g'}^g: \Sigma^g M = \Spin(M,g) \times_{\sigma} \Sigma_n
&\to 
\Spin(M,g') \times_{\sigma} \Sigma_n = \Sigma^{g'}M \\ 
\psi = [s,\phi]
&\mapsto 
\beta_{g'}^g\psi = [\beta_{g'}^g s, \phi]
\end{align*}
where $(\sigma, \Sigma_n)$ is the complex spinor representation.
The map $\beta_{g'}^g$ preserves fiberwise the length of the spinors. 
 
Let the Dirac operator $D^{g'}$ act on sections of $\Sigma^g M$ as the
operator
\[ 
{}^{g\mkern-4mu} D^{g'}
\definedas 
( \beta_{g'}^g)^{-1} \circ D^{g'} \circ \beta_{g'}^g.
\]
Compared with the Dirac operator $D^g$ on $\Sigma^g M$ there
is the following relation, see 
\cite[Th\'eor\`eme 20]{bourguignon_gauduchon_92},
\begin{equation} \label{compare_spinors_BG}
{}^{g\mkern-4mu} D^{g'}
= 
D^g \psi + A_{g'}^g(\nabla^g \psi) + B_{g'}^g (\psi)
\end{equation}
where $A_{g'}^g$ and $B_{g'}^g$ are pointwise vector bundle maps whose
pointwise norms are bounded by 
\begin{equation} \label{boundAB}
| A_{g'}^g | \leq C |g - g'|_g ,
\qquad
| B_{g'}^g | \leq C( |g - g'|_g + |\nabla^g(g - g')|_g ) .
\end{equation}

When $g'$ and $g$ are conformal with $g' = F^2 g$ for a positive
smooth function $F$ we have
\begin{equation} \label{Dconfchange} 
{}^{g\mkern-4mu} D^{g'}
(F^{ -\frac{n-1}{2} } \psi )
= 
F^{-\frac{n+1}{2}} D^g \psi .
\end{equation}

\subsection{Removal of singularities}

The next Lemma tells us that a spinor in $L^2$ which is harmonic
outside a subset of codimension two can be extended to a harmonic
spinor everywhere. 

\begin{lemma} \label{removal_singularities}
Let $(M,g)$ be a compact $(n+1)$-dimensional manifold with boundary
$\partial M$, and let $g$ be product on $\partial M \times [-t_0, 0]$. 
Moreover, let $S\subset \partial M$ be a compact submanifold of
dimension $k \leq n-2$. Let the manifold $M_\infty$ be obtained from
$M$ as described above. Let $B \subset M_\infty$ be a submanifold
(possibly with boundary) of dimension $k+1$ with 
$S \times [-t_0,\infty) \subset B$ and such that 
$B \setminus (S \times (-t_0,\infty))$ is a compact submanifold with 
boundary. Assume that $\phi$ is a spinor with
$\Vert \phi \Vert_{L^2(M_\infty)} < \infty$ and $D^g\phi = 0$ weakly on
${M_\infty} \setminus B$. Then $D^g \phi = 0$ holds weakly also on
${M_\infty}$. 
\end{lemma}

Note that the Lemma includes in the case $B = S\times [-t_0,\infty)$.

\begin{proof}
The proof follows the method of 
\cite[Lemma~2.4]{ammann_dahl_humbert_09}. Let
$\psi$ be a compactly supported spinor. We will show that
$\int_{M_\infty} \< \phi, D^g\psi \> \, dv^g = 0$. 

Let $U_B(\delta)$ consist of the points in $M_\infty$ with distance to
$B$ less than $\delta$. Let $\eta: M_\infty \to [0,1]$ be a smooth
cut-off function with $\eta = 1$ on $U_B(\delta)$, $\eta = 0$ on
$M_\infty \setminus U_B(2\delta)$ and $|\grad^g \eta|\leq
2/\delta$. We compute 
\[ \begin{split}
\left| \int_{M_\infty} \< \phi, D^g \psi \> \, dv^g \right| 
&= 
\left| \int_{M_\infty} \< \phi, D^g( (1-\eta)\psi +\eta \psi) \> 
\, dv^g \right| \\ 
&\leq 
\left| \int_{M_\infty} \< \phi, D^g( (1-\eta)\psi) \> \, dv^g \right| 
+ \left| \int_{M_\infty} \< \phi, \eta D^g \psi \> \, dv^g \right| \\
&\qquad
+ \left| \int_{M_\infty} \< \phi, \grad^g \eta \cdot \psi \> \, dv^g
\right| \\
&\leq 
\left| \int_{M_\infty\setminus U_B(\delta)} \!\!
\< \phi, D^g( (1-\eta)\psi) \> \, dv^g \right|
+ \Vert \phi \Vert_{L^2(U_B(2\delta))} 
\Vert D^g \psi\Vert_{L^2} \\
&\qquad 
+ \frac{2}{\delta}
\Vert\phi\Vert_{L^2(U_B(2\delta))}
\Vert\psi\Vert_{L^2(U_B(2\delta))}.
\end{split} \] 
The first term vanishes since $D^g\phi = 0$ weakly on
$M_\infty \setminus B$ and $(1-\eta)\psi$ is compactly supported on 
$M_\infty\setminus B$. The second summand goes to $0$ as 
$\delta \to 0$. To estimate the third term note that 
\[ \begin{split}
\Vert \psi\Vert^2_{L^2(U_B(2\delta))}
&\leq \max |\psi|^2 \vol (U_B(2\delta)\cap \supp \psi)\\
&\leq \max |\psi|^2 C(\psi) \vol_{k+1} (\mathrm{B}\supp \psi) 
(2\delta)^{{n-k}}
\end{split} \] 
where $\vol_k$ measures the $k$-dimensional volume, $C(\psi)>0$ and
$\mathrm{B}\supp\psi$ denotes a compact subset of $B$ such that
$(U_B(2\delta)\cap \supp \psi) \subset 
U_{\mathrm{B}\supp\psi}(2\delta)$. Then, 
\[ 
\frac{2}{\delta}\Vert \phi\Vert_{L^2(U_B(2\delta))}
\Vert \psi\Vert_{L^2(U_B(2\delta))}
\leq C \delta^{\frac{n-k}{2}-1} \Vert\phi\Vert_{L^2(U_B(2\delta))}
\] 
where $C$ only depends on $\psi$ and, thus, with $ n-k \geq 2$ this
term also tends to $0$ as $\delta\to 0$. 
\end{proof}

\section{Handle attachment} 

In this section the proof of Theorem~\ref{thm_handle_attachment} is
given in a sequence of steps. We begin by giving an overview and
explaining the strategy of the proof. 

\subsection{Overview of the proof}
\label{strategy}

We will use a similar construction as Carr in \cite{carr_88} where it
is proved that the existence of positive scalar curvature metrics on
manifolds with boundary is preserved under handle attachment of
codimension $\geq 3$. For this the manifold is doubled in order to 
obtain a closed manifold and the handle attachment construction is 
split into two steps to make the construction of the new metric easier. 

We will also split the surgery into two steps, but we work with the 
original manifold with attached cylindrical ends since we are 
interested in the invertibility of the Dirac operator.

We now describe the topological construction, and then explain how
the metric will be obtained. 

\subsubsection{Topological strategy} 

Let $(M,g)$ be the initial manifold with product structure near the
boundary on $(-t_0,0] \times \partial M$. Moreover let $(M_\infty, g)$
be $M$ with cylindrical ends attached, and let $S \subset \partial M$
be the handle attachment sphere, where $S$ is diffeomorphic to $S^k$.

\begin{itemize}
\item 
First we construct a surgery along $S^k\times B^{n-k}\times
(-\epsilon, \epsilon) \hookrightarrow \partial M \times
(-t_0, \infty) \subset M_\infty$ where $S^k\times \{0\}$ is mapped to
$S$, see the first picture in Figure~\ref{figure_strategy_a} where $S$
is indicated as the dots inside the circles. By replacing the
image of $S^k\times B^{n-k}\times (-\epsilon, \epsilon)\cong S^k\times
B^{n-k+1}$ by $B^{k+1}\times S^{n-k}$ we obtain $M'_\infty$. 
\item
Second we embed $S^k\times B^{n-k}\times (c, \infty)$ into the part
of $M'_\infty$ which lies ``above'' the first surgery, that is in 
$\partial M\times (c,\infty)$ for certain $c$. Moreover, we embed 
$B^{k+1} \times B^{n-k}$ into the attached handle $B^{k+1}\times S^{n-k}$ 
of the first surgery. Gluing both along its part of the boundary lying
in $\partial M$, that is $S^k\times B^{n-k}\subset \partial M$, we obtain
an embedding
\[
S^k\times B^{n-k}\times (c,\infty) \sqcup B^{k+1}\times B^{n-k}
\cong 
B^{k+1} \times B^{n-k} \hookrightarrow M'_\infty
\]
The second surgery will replace the embedded $B^{k+1} \times B^{n-k}$
by $B_-^{k+2}\times S^{n-k-1}$. 
\end{itemize}

\begin{figure}[h]
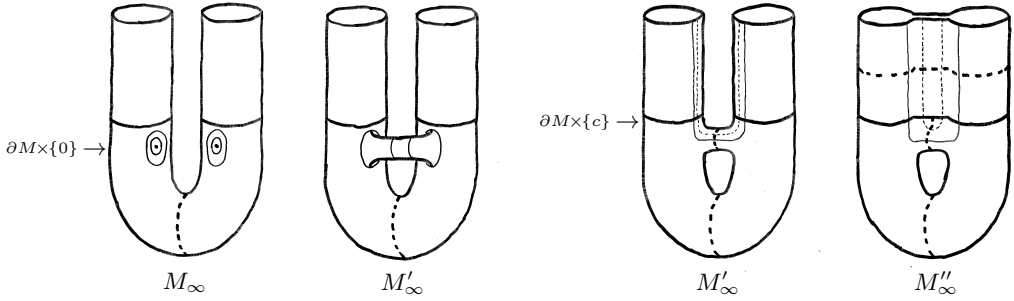

\begin{minipage}{0.44\textwidth}
\begin{lpic}[]{FIGS/STEP2_A(0.20)}
\lbl[]{75,7;{\small $M_{\infty}$}}
\lbl[]{220,7;{\small $M'_{\infty}$}}
\lbl[]{-20,96;{\tiny $\partial M\!\!\times\!\!\{0\}$}}
\lbl[]{15,95;{\small $\rightarrow$}}
\end{lpic}
\end{minipage}\hfill
\begin{minipage}{0.44\textwidth}
\begin{lpic}[]{FIGS/STEP4_A(0.20)}
\lbl[]{75,7;{\small $M'_{\infty}$}}
\lbl[]{220,7;{\small $M''_{\infty}$}}
\lbl[]{-20,114;{\tiny $\partial M\!\!\times\!\!\{c\}$}}
\lbl[]{15,113;{\small $\rightarrow$}}
\end{lpic}
\end{minipage}
\caption{Surgery divided in two steps}
\label{figure_strategy_a}
\end{figure}

Note that after cutting $M'_\infty$ along the former boundary of $M$ 
(which is $\partial M \times \{0\}\subset M_\infty$) we already get
the desired surgery, $S^k\times B_{-}^{n-k+1}$ is replaced by
$B^{k+1}\times S_{-}^{n-k+1}$. Thus, topologically this would
suffice. But in order to obtain a metric which has product
structure near the boundary we have the second surgery which produces
a cylinder above the boundary and which does not change the topology
below the boundary. Thus, after both steps we still have the desired
surgery on the manifold with boundary and additionally we already got
the corresponding manifold with attached cylinders. 

\subsubsection{Metric strategy} 

One of the main tasks in the proof is to construct approximations of 
the metric such that the handles can be glued into the manifold and 
such that the metrics are easily extended to the handles. Moreover, 
this has to be done in such a way that the new metrics can be chosen
to be arbitrarily close to the old one but still have an invertible 
Dirac operator on the manifold before and after surgery.

We now explain the steps in the proof.

\begin{itemize}
\item
In Step 1, before starting with the first surgery, we approximate
$g$ by metrics $\overline{g}_\delta$. The new metrics will have product 
structure on a small tubular neighbourhood of
$S\times (-t_0+\delta, \infty)$. This product structure is not only a 
product in the direction tangential to the boundary as before but also 
product of $S$ and the normal directions inside the boundary. Moreover, 
the new metric will coincide with $g$ outside a larger tubular 
neighbourhood $S\times (-t_0+\delta, \infty)$, see 
Figure~\ref{figure_step1a}. In Proposition~\ref{first_approx_prop} we
show that choosing $\delta$ small enough, the metrics 
$\overline{g}_\delta$ will still have invertible Dirac operators. 
In the proof, one can easily rule out the case that zero is in the 
essential spectrum by using that the induced metric on the boundary 
will still have invertible Dirac operator which gives a spectral gap 
on the cylindrical end, see Lemma~\ref{lemmabdrygap} and 
Proposition~\ref{propbdrygap}, Part (1). Thus, the main task will be 
to show that zero is not an eigenvalue for $\overline{g}_\delta$ for 
$\delta$ small enough which will be done by estimating the norm of the 
spinor at the cylindrical end using Proposition~\ref{propbdrygap}, 
Part (3). 
\item
In Step 2 the first surgery is performed. The product structure 
produced in step 1 allows us to obtain a metric $g'$ on $M_\infty'$. 
In Proposition~\ref{first_surgery_prop} we show that $g'_\rho$ has 
invertible Dirac operator if the size of the surgery, as measured by 
$\rho$, is sufficiently small. The proof is similar to the one of
Proposition~\ref{first_approx_prop} in Step 1 since the cylindrical
ends are not affected by the construction. But now one has to exclude 
that the norm of the harmonic spinor is concentrated in the attached 
neck. This will be done by an a priori estimate, see 
Lemma~\ref{fundamental_estimate_first_n}.
\item
Step 3 is a second approximation of the metric. Above the first surgery, 
that is on a neighbourhood of $S\times [c, \infty)$, $g'$ still has the 
desired product structure. The aim of this step is to extend 
$S\times [c, \infty)$ and therewith the product structure of its 
neighbourhood smoothly to a neighbourhood of 
$B^{k+1}\times \{ {\rm pt. } \}$ sitting in the attached handle. This 
gives a product structure on the neighbourhood of 
$B^{k+1} \equiv S^k \times [c, \infty) \sqcup 
B^{k+1}\times \{ {\rm pt.} \} \hookrightarrow M'_\infty$. Again,
choosing the involved parameters sufficiently small the resulting
metric has invertible Dirac operator, see 
Proposition~\ref{Second_approx_prop}.
\item
In Step 4 the second surgery is done and results in the desired metric
$g'' \in \Rinv(M)$ for sufficiently small surgery parameter. This
will be proved in Proposition~\ref{second_surgery_prop} and as in Step 2 
an additional estimate (see Lemma~\ref{fundamental_estimate_first}) is 
needed to ensure that the norm of the spinor is not concentrated in the 
infinite part attached by the surgery.
\end{itemize}

\subsubsection{Notation}

Before starting the proof we need to introduce refined notation for
the surgery embedding as in the beginning of Section~3 in
\cite{ammann_dahl_humbert_09}. Let $(M,g)$ be a compact spin manifold
with boundary. The manifold $M''$ is obtained from $M$ by
surgery using the embedding $f: S^k \times B^{n-k} \to \partial M$. We
now make some more detailed assumptions about the map $f$.

Let $i: S^k \to \partial M$ be an embedding and set
$S \definedas i(S^k)$. Let $\pi^\nu : \nu \to S$ be the normal bundle
of $S$ in $(\partial M, \partial g)$. We assume that a trivialization
of $\nu$ is given through a vector bundle map 
$\iota : S^k \times \mR^{n-k} \to \nu$ such that 
$(\pi^{\nu} \circ \iota)(p,0) = i(p)$ for $p \in S^k$. Further we
assume that $\iota$ is fiberwise an isometry when the fibers
$\mR^{n-k}$ of $S^k \times \mR^{n-k}$ are given the standard metric,
and the fibers of $\nu$ have the metric induced by $\partial g$. We
get the embedding $f$ by setting $f \definedas \exp^{\nu} \circ \iota:
S^k \times B^{n-k}(R) \to \partial M$ for sufficiently small $R$. 
We define open neighborhoods $U_S(R)$ of $S$ in $\partial M$ by
\[
U_S(R) \definedas
(\exp^{\nu} \circ \iota) ( S^k \times B^{n-k}(R) )
\]
for $R$ small enough. For a point $x \in \partial M$ set 
$r(x) \definedas d^g(x,S)$ to be the distance from $x$ to $S$. Again,
let $h$ denote the pullback by $i$ to $S^k$ of the restriction of $g$
to the tangent bundle of $S$, 
\[
h \definedas
i^* ( g |_{TS \times TS} ).
\]
Our goal is to perturb the metric $g$ slightly so that
the map $f$ becomes an isometry if its domain is equipped with the
product metric $h + \xi^{n-k}$. The next lemma gives an estimate of how
much this fails for the metric $g$.
\begin{lemma}[{\cite[Lemma~3.1]{ammann_dahl_humbert_09}}]
\label{ADHlemma3.1}
For sufficiently small $R > 0$ there is a constant $C > 0$ so that 
\[
G \definedas
\partial g - (f^{-1})^* (h + \xi^{n-k})
\]
satisfies 
\[
|G| \leq C r, \qquad |\nabla G| \leq C 
\]
on $U_S(R)$.
\end{lemma}

We are now ready to go through the steps of the proof.

\subsection{Step 1: Approximating by product metrics}
\label{step1}

We show that the metric on $(M,g)$ can be perturbed to have product
form near the surgery sphere, the argument follows
\cite[Proposition~3.2]{ammann_dahl_humbert_09}. We recall that the
metric $g$ has by assumption a cylindrical structure 
$g = \partial g + dt^2$ in a neighbourhood 
$\partial M \times (-t_0, 0]$ of the boundary.
 
\begin{prop} \label{first_approx_prop}
The metric $g \in \Rinv(M)$ can be arbitrarily closely approximated by
metrics $\overline{g}_{\delta} \in \Rinv(M)$ which have 
\[
\overline{g}_{\delta} 
= 
\partial \overline{g}_{\delta} + dt^2 = h + \xi^{n-k} + dt^2
\]
on $U_S(\delta) \times (-t_0 + 2\delta, \infty)$ and
\[
\overline{g}_{\delta} = g
\]
outside $U_S(2\delta) \times (-t_0 + \delta, \infty)$.
\end{prop}

Before discussing the proof of this Proposition we define the metrics
$\overline{g}_{\delta}$. 

Let $\chi: \mR \to [0,1]$ be a smooth decreasing function with 
$\chi = 1$ on $(-\infty, 1]$, $\chi = 0$ on $[2,\infty)$, and 
$-2 \leq \chi' \leq 0$. On the part of $(M_\infty,g)$ which is
isometric to $(\partial M \times (- t_0,\infty), \partial g + dt^2 )$
we define a cut-off function
\[
\eta(x,t) \definedas 
\chi(r(x)/ \delta) (1 - \chi( (t + t_0)/ \delta ) ),
\]
where $\delta > 0$ is a small parameter. This has the property that
$\eta(x,t) = 1$ if $x \in U_S(\delta)$ and 
$t \geq - t_0 + 2\delta$, and $\eta(x,t) = 0$ if 
$x \in \partial M \setminus U_S(2\delta)$ or if $t \leq -t_0 + \delta$.
We define the metrics
\[
\overline{g}_{\delta}
\definedas
\eta (f^{-1})^* (h + \xi^{n-k}) + (1 - \eta) \partial g + dt^2
\]
on $\partial M \times (- t_0,\infty)$ and we extend them by setting
$\overline{g}_{\delta} \definedas g$ on the rest of $M$. The metric 
$\overline{g}_{\delta}$ has the required product structure where 
$\eta = 1$, that is on 
$U_S(\delta) \times (- t_0 + 2\delta, \infty)$. Further, we have 
$\overline{g}_{\delta} = g$ outside 
$U_S(2\delta) \times (- t_0 + \delta, \infty)$. From
\[
\overline{g}_{\delta} - g
=
\eta ( (f^{-1})^* (h + \xi^{n-k}) - \partial g )
\]
together with \eqref{boundAB} and Lemma~\ref{ADHlemma3.1} we get that 
\begin{equation} \label{ABleqr}
|A^g_{\overline{g}_{\delta}}| \leq C \eta r, 
\qquad 
|B^g_{\overline{g}_{\delta}}| \leq C \eta + C |\grad^g \eta| r
\end{equation}
for a some $C>0$. The metric $\overline{g}_{\delta}$ restricted to the
boundary $\partial M$ gives the boundary metric
\[ \begin{split}
\partial \overline{g}_{\delta}
&=
\eta (f^{-1})^* (h + \xi^{n-k}) + (1 - \eta) \partial g \\
&= 
\chi(r / \delta) (f^{-1})^* (h + \xi^{n-k}) 
+ (1 - \chi(r / \delta) ) \partial g .
\end{split} \]

Figure~\ref{figure_step1a} shows $(M_{\infty}, g)$ to the left and 
$(M_{\infty}, \overline{g}_{\delta})$ with the product region shaded
to the right.

\begin{figure}[h]
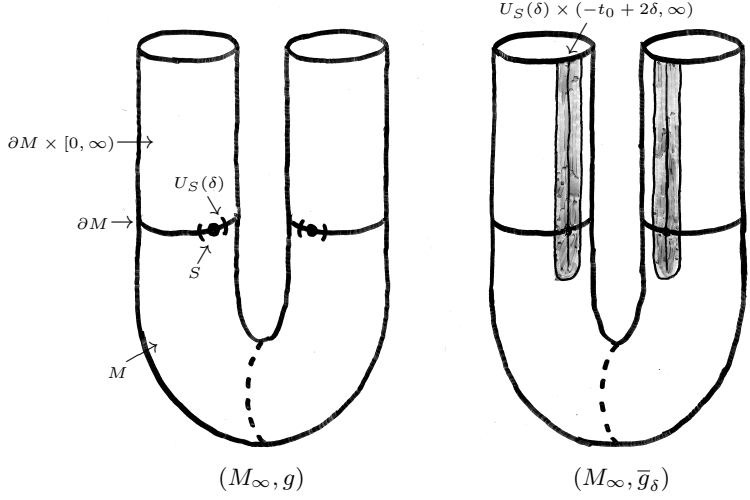

\begin{lpic}[]{FIGS/STEP1_A(0.33)}
\lbl[]{75,10;{\small $(M_{\infty}, g)$}}
\lbl[]{220,10;{\small$(M_{\infty}, \overline{g}_{\delta})$}}

\lbl[]{17,53;{\tiny $M$}}
\lbl[]{27,60,30;{\small $\longrightarrow$}}
\lbl[]{7,113;{\tiny $\partial M$}}
\lbl[]{19,113;{\small $\rightarrow$}}
\lbl[r]{16,145;{\tiny $\partial M \times [0,\infty)$}}
\lbl[]{24,145;{\small $\longrightarrow$}}

\lbl[]{48,93;{\tiny $S$}}
\lbl[]{52,100,60;{\small $\rightarrow$}}
\lbl[]{50,129;{\tiny $U_S(\delta)$}}
\lbl[]{55,120,-60;{\small $\rightarrow$}}
\lbl[]{210,199;{\tiny $U_S(\delta) \times (-t_0 + 2\delta, \infty)$}}
\lbl[]{202,188,-120;{\small $\longrightarrow$}}
\end{lpic}
\caption{Approximating with a product metric.}
\label{figure_step1a}
\end{figure}

We begin by proving that the boundary metrics have uniform spectral
gaps around zero. The proof is very similar to 
\cite[Proposition~3.2]{ammann_dahl_humbert_09}.

\begin{lemma} \label{lemmabdrygap}
There are constants $\Lambda, \delta_0>0$ such that the Dirac operator
of the closed manifold $(\partial M, \partial \overline{g}_{\delta})$
has a spectral gap $(-\Lambda, \Lambda)$ for all $\delta < \delta_0$.
\end{lemma}

\begin{proof}
We argue by contradiction and assume that there is a sequence
$\delta_i \to 0$ such that 
\[
D^{\partial \overline{g}_{\delta_i}} \phi_i = \lambda_i \phi_i
\]
where $\lambda_i \to 0$ and $\phi_i$ are spinors on 
$(\partial M, \partial \overline{g}_{\delta_i})$ with 
$\int_{\partial M} |\phi_i|^2 \, dv^{\partial \overline{g}_{\delta_i}} = 1$.
The proof continues exactly as in \cite{ammann_dahl_humbert_09} and 
uses \eqref{compare_spinors_BG} and \eqref{ABleqr}. 
\end{proof}

We are now ready to prove Proposition~\ref{first_approx_prop}.

\begin{proof}[Proof of Proposition~\ref{first_approx_prop}] 
The metrics $\overline{g}_{\delta}$ have the required product
structure, we need to show that $\overline{g}_{\delta} \in \Rinv(M)$
when $\delta$ small enough. We proceed by assuming the contrary: there
exists a sequence $\delta_i \to 0$ such that the operators
$D^{\overline{g}_{\delta_i}}$ are not invertible. 

From Lemma~\ref{lemmabdrygap} we know that there are constants 
$\Lambda, \delta_0>0$ such that the restriction 
$\partial g_\delta \definedas g_\delta|_{\partial M\times \{0\}}$ has
a spectral gap $(-\Lambda, \Lambda)$ for all $\delta < \delta_0$. 
From Proposition~\ref{propbdrygap} it then follows that the essential
spectrum of $D^{\overline{g}_{\delta_i}}$ has the same gap, and, thus,
the non-invertibility comes from a zero eigenvalue. Hence, there 
is a sequence of $L^2$-spinors $\phi_i$ on 
$(M_{\infty}, \overline{g}_{\delta_i})$ with 
$D^{\overline{g}_{\delta_i}}\phi_i = 0$ and 
$\int_{M_\infty} |\phi_i|^2 \, dv^{\overline{g}_{\delta_i}} = 1$. 
First, we note that $\overline{g}_{\delta_i} = g$ on 
$M_\infty \setminus (U_S(2\delta_i)\times [-t_0+\delta_i,\infty))$. 
Set $U(\delta) \definedas U_S(\delta)\times (-t_0, \infty)$.

Fix $\gamma > 0$. Then for all $i$ with $2\delta_i<\gamma$ and all
compact subsets $K$ of $M_\infty \setminus U(\gamma)\subset
M_\infty\setminus (U_S(2\delta_i)\times [-t_0+\delta_i,\infty))$
we have from Lemma~\ref{C^2-estimate} that there is a 
constant $C = C(K, M_\infty \setminus U(\gamma),g)$ with 
\[
\Vert\phi_i\Vert_{C^2(K)}
\leq 
C\Vert \phi_i 
\Vert_{L^2(M_\infty \setminus U(\gamma),g)}
\leq C.
\]
From the Theorem of Ascoli, Lemma~\ref{lemma_Ascoli}, we
obtain that $\phi_i\to \phi$ strongly in $C^1(K)$ and $D^g\phi=0$
weakly on each $K$. Moreover, $\phi_i \to \phi$ weakly in
$L^2(M_\infty \setminus U(\gamma), g)$ and 
$\Vert \phi \Vert_{L^2(M_\infty\setminus U(\gamma), g)} \leq 1$. 
Thus, if $\gamma \to 0$ we obtain that $D^g\phi=0$ weakly on 
$M_\infty \setminus (S \times [-t_0,\infty))$ and $\phi \in
L^2(M_\infty,g)$. From Lemma~\ref{removal_singularities} we then 
have $D^g\phi=0$ weakly on $M_\infty$. 

It remains to show that $\phi$ is not identically zero. We prove this
by contradiction and assume that $\phi = 0$. Thus, due to the
Rellich-Kondrakov Theorem $\phi_i \to 0$ in $L^2(g)$ on compact
subsets. In particular, 
$\int_K |\phi_i|^2 \, dv^{\overline{g}_{\delta_i}} \to 0 $ as $i\to\infty$
for each compact $K \subset M_\infty$, since
$|\overline{g}_{\delta_i} - g|\to 0$ on compact $K$. 

To study $\phi_i$ on $\partial M \times (0,\infty)$ we set
$l_i(t)^2 \definedas \int_{\partial M\times\{t\}} |\phi_i|^2 
\, dv^{\partial \overline{g}_{\delta_i}}$. From part (3) of Proposition
\ref{propbdrygap} we have $l_i(s)^2 \leq 
(2\Lambda)^{-1} e^{-2\Lambda (s-1)} \int_0^1 l_i(t)^2 \, dt$. 
Integrating this gives us 
\[
\int_1^\infty l_i(s)^2 \, ds 
\leq 
\frac{1}{2\Lambda} 
\int_1^\infty e^{-2\Lambda(s-1)} \, ds 
\int_0^1 l_i(t)^2 \, dt
=
\frac{1}{4\Lambda^{2}} \int_0^1 l_i(t)^2 \, dt 
\]
and, thus,
\[ \begin{split}
1 &=
\int_{M_\infty} |\phi_i|^2 \, dv^{\overline{g}_{\delta_i}} \\
&=
\int_{M} |\phi_i|^2 \, dv^{\overline{g}_{\delta_i}} 
+ \int_0^{\infty} l_i(t)^2 \, dt\\ 
&\leq 
\int_{M} |\phi_i|^2 \, dv^{\overline{g}_{\delta_i}} 
+ \left( 1 + \frac{1}{4\Lambda^{2}} \right) 
\int_0^1 l_i(t)^2 \, dt\\ 
&\leq 
\left( 1 + \frac{1}{4\Lambda^{2}} \right) 
\int_{M\cup (\partial M\times [0,1])} 
|\phi_i|^2 \, dv^{\overline{g}_{\delta_i}} , 
\end{split} \]
which gives a contradiction since $\phi_i$ is supposed to tend to zero
in $L^2(g)$ on the compact set $M \cup (\partial M \times [0,1])$. 
Thus, we obtained a nontrivial $L^2(g)$-harmonic spinor $\phi$ on
$(M_\infty, g)$ which contradicts the assumption that $g\in \Rinv(M)$. 
\end{proof}

After the first step we replace $g$ by $\overline{g}_{\delta_0}$ for
some $\delta_0$ sufficiently small, we also set 
$-t_1 \definedas -t_0 + 2\delta_0$ and 
$R_{\rm max} \definedas \delta_0$ 
(and perhaps we make the spectral gap of $D^g$ a bit smaller). The
conclusion of this first step is then that we may assume that 
\[
\overline{g}_{\delta_0} = \partial g + dt^2 = h + \xi^{n-k} + dt^2
\]
on $U_S( R_{\rm max} ) \times (-t_1, \infty)$ and
\[
\overline{g}_{\delta_0} = g
\]
outside $U_S(2 R_{\rm max} ) \times (-t_1 -R_{\rm max}, \infty)$,
while the spectral gap is the same.

From now on we assume that the metric $g$ has already the form
$\overline{g}_{\delta_0}$. 

\subsection{Step 2: First surgery}

We now perform a standard surgery of codimension $n-k+1$ on
$(M_\infty,g)$ along the embedding 
\[ \begin{aligned}
S^k \times B^{n-k+1} 
= S^k \times B^{n-k} \times (-\epsilon, \epsilon) 
&\to \partial M \times (-t_1, \infty) \subset M_{\infty} 
\\ 
f^{\rho}: (x,y,t) &\mapsto (f(x,y), - 2\rho + t)
\end{aligned} \]
where $-t_1<-2\rho-\epsilon<-2\rho+\epsilon<-\rho$. Here $\rho$ is a 
parameter which will be specified later, and the first equality in 
the embedding comes from the choice of a diffeomorphism 
$B^{n-k+1} \simeq B^{n-k} \times (-\epsilon, \epsilon)$. 

Denote by $M'_\infty$ the resulting manifold after surgery and by $M'$
the same manifold without the cylindrical end. We will construct a
family of metrics $g'_{\rho}$ on $M'_\infty$ which coincide with $g$
outside the distance $\rho$ of the surgery sphere. 

On $U_S( R_{\rm max} ) \times (-t_1, \infty)$ the metric $g$ has the
product form 
\[
g = \partial g + dt^2 = h + \xi^{n-k} + dt^2 = h + \xi^{n-k+1}.
\]
The surgery in this step is centered around the surgery sphere
$S_{\rho} \definedas S \times \{ -2\rho \} 
\subset \partial M \times (-t_1, 0]$. We write the flat metric
$\xi^{n-k+1}$ in polar coordinates around 
$(0,-2\rho) \in B^{n-k} (R_{\rm max}) \times (-t_1, 0]$, and we get
\[
g = h + d\hat{r}^2 + \hat{r}^2 \sigma^{n-k} 
\]
where $\hat{r} = \sqrt{r^2 + (t+2\rho)^2}$ is the distance to the
point $(0,-2\rho)$ and $r$ is the distance to 
$S \times (-t_0, \infty)$. Set 
$U_{S_{\rho}}(R) \definedas \{ \hat{r} \leq R \}\subset M_\infty$. 
Figure~\ref{figure_step2c} shows the placement of $S_{\rho}$ and 
$U_{S_{\rho}}(R)$.
\begin{figure}[h]
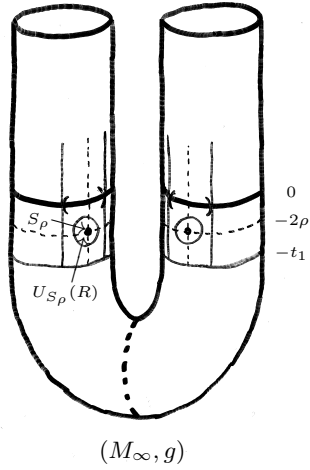

\begin{lpic}[]{FIGS/STEP2_C(0.33)}
\lbl[]{75,10;{\small $(M_{\infty}, g)$}}
\lbl[]{135,115;{\tiny $0$}}
\lbl[]{135,103;{\tiny $-2\rho$}}
\lbl[]{135,90;{\tiny $-t_1$}}
\lbl[]{33,103;{\tiny $S_{\rho}$}}
\lbl[]{45,101,-15;{\small $\longrightarrow$}}
\lbl[]{44,73;{\tiny $U_{S_{\rho}}(R)$}}
\lbl[]{46,88,75;{\small $\xrightarrow{\phantom{XX}}$}}
\end{lpic}
\caption{The surgery sphere $S_{\rho}$.}
\label{figure_step2c}
\end{figure}

We divide $M$ into three pieces:
\begin{itemize}
\item[$\{A\}$] 
$M \setminus U_{S_{\rho}}(R_{\rm max}/2)$
\item[$\{B\}$] 
$U_{S_{\rho}}(R_{\rm max}/2) \setminus U_{S_{\rho}}(\rho /2)
\simeq
S^k \times (\rho/2, R_{\rm max}/2) \times S^{n-k} $
\item[$\{C\}$]
$U_{S_{\rho}}(\rho /2)
\simeq
S^k \times B^{n-k+1} (\rho/2)$
\end{itemize}
The manifold $M'$ after surgery is obtained by replacing
$\{C\}$ by 
\begin{itemize}
\item[$\{C'\}$]
$B^{k+1} \times S^{n-k}$,
\end{itemize}
see Figure~\ref{figure_step2a}.
\begin{figure}[h]
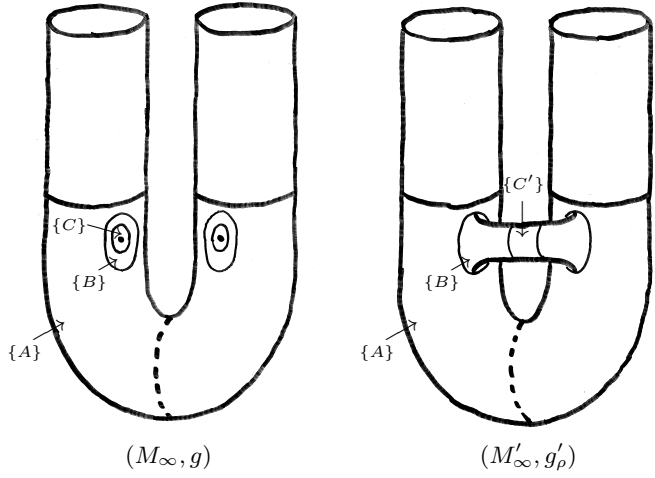

\begin{lpic}[]{FIGS/STEP2_A(0.33)}
\lbl[]{75,10;{\small $(M_{\infty}, g)$}}
\lbl[]{220,10;{\small $(M'_{\infty}, g'_{\rho})$}}
\lbl[]{17,52;{\tiny $\{A\}$}}
\lbl[]{27,60,30;{\small $\longrightarrow$}}
\lbl[]{43,80;{\tiny $\{B\}$}}
\lbl[]{51,87,30;{\small $\rightarrow$}}
\lbl[]{35,103;{\tiny $\{C\}$}}
\lbl[]{49,101,-15;{\small $\longrightarrow$}}
\lbl[]{159,52;{\tiny $\{A\}$}}
\lbl[]{169,60,30;{\small $\longrightarrow$}}
\lbl[]{185,80;{\tiny $\{B\}$}}
\lbl[]{193,87,30;{\small $\rightarrow$}}
\lbl[]{218,120;{\tiny $\{C'\}$}}
\lbl[]{217,106,-90;{\small $\longrightarrow$}}
\end{lpic}
\caption{First surgery.}
\label{figure_step2a}
\end{figure}

We define metrics $g'_{\rho}$ on $M'$ by 
\begin{itemize}
\item[$\{A\}$] 
$g'_{\rho} \definedas g$
\item[$\{B\}$] 
$g'_{\rho} \definedas 
h + d\hat{r}^2 + \alpha_{\rho}(\hat{r})^2 \sigma^{n-k} $
\item[$\{C'\}$]
$g'_{\rho} \definedas
H + (2\rho/3)^2 \sigma^{n-k} $
\end{itemize}
where the function $\alpha_{\rho}$ is as in 
Figure~\ref{figure_alpharho}
\begin{figure}[h]
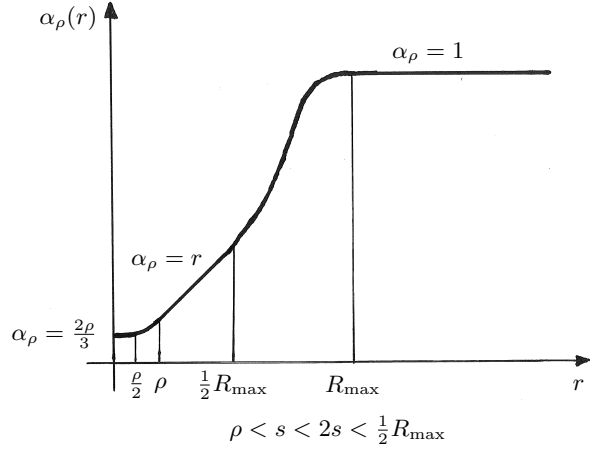

\begin{lpic}[]{FIGS/STEP2_B(0.4)}
\lbl[r]{25,137;{\small $\alpha_{\rho}(r)$}}
\lbl[r]{25,32;{\small $\alpha_{\rho} = \frac{2\rho}{3} $}}
\lbl[br]{60,53;{\small $\alpha_{\rho} = r$}}
\lbl[b]{135,122;{\small $\alpha_{\rho} = 1$}}
\lbl[]{38,15;{\small $\frac{\rho}{2}$}}
\lbl[]{46,15;{\small $\rho$}}
\lbl[]{70,15;{\small $\frac{1}{2}R_{\rm max}$}}
\lbl[]{110,15;{\small $R_{\rm max}$}}
\lbl[]{185,15;{\small $r$}}
\lbl[]{105,0;{\small $\rho < s < 2s < \frac{1}{2} R_{\rm max}$}}
\end{lpic}
\caption{The function $\alpha_{\rho}$.}
\label{figure_alpharho}
\end{figure}
and $H$ is a metric on $B^{k+1}$ which is equal to $d\hat{r}^2 + h$
near the boundary which is possible since near the boundary $B^{k+1}$
is diffeomorphic to $S^k\times [0,\epsilon]$. 
Figure~\ref{figure_step2a} shows $(M_\infty,g)$ to the left and 
$(M'_\infty,g'_{\rho})$ after surgery to the right.

Define the subset $U'(R) \subset M'_\infty$ by 
$M'_\infty \setminus U'(R) = M_\infty \setminus U_{S_{\rho}}(R)$ for 
$R>\frac{\rho}{2}$. Note that $\alpha_\rho (\hat{r})=\hat{r}$ on 
$\left[ \rho, \frac{1}{2} R_{\rm max} \right]$ and, thus, 
$g'_{\rho} = g$ on 
$M'_\infty \setminus U'(\rho)$. Note also that the definition of 
$g'_{\rho}$ does not involve $\alpha_{\rho}(\hat{r})$ for 
$\hat{r} > R_{\rm max}/2$. This part is defined so that we easily can
extend the function $\alpha_{\rho}$ to all of $M'_\infty$. We set
$\alpha_{\rho} = 1$ on $M'_\infty \setminus U'(R_{\rm max})$ and
$\alpha_{\rho} = 2\rho/3$ on $\{C'\}$.

We define a conformally related metric on $M'$ by
\[
\widetilde{g}_{\rho} \definedas \alpha_{\rho}^{-2} g'_{\rho}.
\]
On $\{B\} + \{C'\}$ we have that $\widetilde{g}_{\rho}$ is a product 
metric,
\[
\widetilde{g}_{\rho} = \alpha_{\rho}^{-2} H + \sigma^{n-k},
\]
where $H$ is defined as $d\hat{r}^2 + h$ on $\{B\}$. 
For the proof of Proposition~\ref{first_surgery_prop} we need the
following Lemma, similar to 
\cite[Proposition~3.5]{ammann_dahl_humbert_09}. 

\begin{lemma} \label{fundamental_estimate_first_n}
Let $s$ be such that $\rho < s < 2s < R_{\rm max} /2$ and assume that 
$D^{g'_{\rho}} \psi'_{\rho} = 0$. Then 
\[
\int_{U'(2s) \setminus U'(s)} | \psi'_{\rho} |^2 \, dv^{g'_{\rho}}
\geq
\frac{1}{8} 
\int_{U'(s)} |\psi'_{\rho}|^2 \, dv^{g'_{\rho}}.
\]
\end{lemma}

\begin{proof}
We make the conformal change 
$\widetilde{g}_{\rho} \definedas \alpha_{\rho}^{-2} g'_{\rho}$
and set 
\[
\widetilde{\psi}_{\rho} 
\definedas
\alpha_{\rho}^{\frac{n}{2}} \beta^{g'_{\rho}}_{\widetilde{g}_{\rho}}
\psi'_{\rho} ,
\]
observe here that we are working on the manifold $M$ which has
dimension $n+1$. From \eqref{Dconfchange} we then have
\[
D^{\widetilde{g}_{\rho}} \widetilde{\psi}_{\rho} = 0.
\]
Choose a cut-off function $\eta$ on $M'$ with 
$\eta = 1$ on $U'(s)$, $\eta = 0$ on $U'(2s)$. 
Since $d\eta$ is supported in $M'\setminus U'(\rho)$ where
$g'_\rho=g$ we may assume
\[
|d\eta|_{g'_{\rho}} \leq 2/s
\]
which implies 
\[ 
|d\eta|^2_{\widetilde{g}_{\rho}} 
= 
\alpha_{\rho}^2 |d\eta|^2_{g'_{\rho}}
\leq 
4 \alpha_{\rho}^2 /s^2.
\]
We have
\[
D^{\widetilde{g}_{\rho}} (\eta \widetilde{\psi}_{\rho} )
= 
\grad^{\widetilde{g}_{\rho}} \eta \cdot \widetilde{\psi}_{\rho}
\]
which is supported in $U'(2s) \setminus U'(s)$ and can be estimated by
\begin{equation} \label{Dwidetildeg}
| D^{\widetilde{g}_{\rho}} (\eta \widetilde{\psi}_{\rho} ) |^2
=
|\grad^{\widetilde{g}_{\rho}} \eta |_{\widetilde{g}_{\rho}}^2 
|\widetilde{\psi}_{\rho}|^2
\leq 
\frac{4 \alpha_{\rho}^2}{s^2} 
|\widetilde{\psi}_{\rho}|^2.
\end{equation}
Since $\widetilde{g}_{\rho} = \alpha_{\rho}^{-2} H + \sigma^{n-k}$ on 
$U'(2s)$ we have a lower spectral bound, see 
\cite[Lemma~2.5]{ammann_dahl_humbert_09} 
\begin{equation} \label{spectralbound}
\begin{split}
\int_{U'(2s)} 
|D^{\widetilde{g}_{\rho}} (\eta \widetilde{\psi}_{\rho} ) |^2 
\, dv^{\widetilde{g}_{\rho}}
&\geq 
\frac{(n-k)^2}{4}
\int_{U'(2s)} 
| \eta \widetilde{\psi}_{\rho} |^2 
\, dv^{\widetilde{g}_{\rho}} \\
&\geq 
\int_{U'(2s)} 
| \eta \widetilde{\psi}_{\rho} |^2 
\, dv^{\widetilde{g}_{\rho}}.
\end{split}
\end{equation}
Using \eqref{Dwidetildeg} we get for the left-hand side,
\[ \begin{split}
\int_{U'(2s)} 
|D^{\widetilde{g}_{\rho}} (\eta \widetilde{\psi}_{\rho} ) |^2 
\, dv^{\widetilde{g}_{\rho}}
&\leq
\frac{4}{s^2} 
\int_{U'(2s) \setminus U'(s)}
\alpha_{\rho}^2 | \widetilde{\psi}_{\rho} |^2 
\, dv^{\widetilde{g}_{\rho}} \\
&=
\frac{4}{s^2} 
\int_{U'(2s) \setminus U'(s)}
\alpha_{\rho} | \psi'_{\rho} |^2 
\, dv^{g'_{\rho}} \\
&\leq
\frac{8}{s} 
\int_{U'(2s) \setminus U'(s)} | \psi'_{\rho} |^2 \, dv^{g'_{\rho}} 
\end{split} \]
where we used that $\alpha_{\rho} \leq 2s$ in the final step, recall
here that $\rho < s$ and, thus, $U'(2s)\setminus U'(s) = 
U_{S_\rho}(2s)\setminus U_{S_\rho}(s)$. Inserted in
\eqref{spectralbound} we get 
\begin{equation} \label{spectralbound2}
\frac{8}{s}
\int_{U'(2s) \setminus U'(s)} | \psi'_{\rho} |^2 \, dv^{g'_{\rho}}
\geq
\int_{U'(2s)} 
| \eta \widetilde{\psi}_{\rho} |^2 
\, dv^{\widetilde{g}_{\rho}}.
\end{equation}
Here we have for the right-hand side,
\[ \begin{split}
\int_{U'(2s)} 
| \eta \widetilde{\psi}_{\rho} |^2 
\, dv^{\widetilde{g}_{\rho}}
&\geq
\int_{U'(s)} 
| \widetilde{\psi}_{\rho} |^2 
\, dv^{\widetilde{g}_{\rho}} \\
&=
\int_{U'(s)} 
\alpha_{\rho}^{-1} | \psi'_{\rho} |^2 
\, dv^{g'_{\rho}} \\
&\geq
\frac{1}{s}
\int_{U'(s)} 
| \psi'_{\rho} |^2 \, dv^{g'_{\rho}} ,
\end{split} \]
where we used that $\alpha_{\rho} \leq s$ in the final step.
Inserted in \eqref{spectralbound2} we get
\[
\frac{8}{s} 
\int_{U'(2s) \setminus U'(s)} | \psi'_{\rho} |^2 \, dv^{g'_{\rho}}
\geq
\frac{1}{s}
\int_{U'(s)} 
| \psi'_{\rho} |^2 \, dv^{g'_{\rho}} ,
\]
or
\[ 
\int_{U'(2s) \setminus U'(s)} | \psi'_{\rho} |^2 \, dv^{g'_{\rho}}
\geq
\frac{1}{8} \int_{U'(s)} | \psi'_{\rho} |^2 \, dv^{g'_{\rho}}
\]
which is the claim of the Lemma. 
\end{proof}

\begin{prop} \label{first_surgery_prop}
$g'_\rho \in \Rinv(M')$ for all sufficiently small $\rho$.
\end{prop}

\begin{proof}
To prove this Proposition we first observe that since the boundary
metric $\partial g'_\rho = \partial g$ is independent of $\rho$ it
follows that the essential spectrum of $D^{g'_\rho}$ has a gap
around zero which is independent of $\rho$, see Proposition
\ref{propbdrygap}. We can then proceed as in the proof of Theorem~1.2
of \cite{ammann_dahl_humbert_09} and assume that there is a sequence
$\rho_i \to 0$ so that $D^{g'_{\rho_i}}$ has a harmonic
spinor $\phi_i \in L^2(M'_\infty, g'_{\rho_i})$. We normalize 
$\int_{M'_\infty} |\phi_i|^2 \, dv^{g'_{\rho_i}} = 1$.

Now let $\delta>0$. For all $\rho_i<\delta$ we have
$M'_{\infty}\setminus U'(\rho_i) = 
M_\infty\setminus U_{S_{\rho_i}}(\rho_i)$ and on this part 
$g_{\rho_i} = g$. Note that
$Z_\delta \definedas M_\infty\setminus U_{S_0}(3\delta) \subset
M'_{\infty}\setminus U'(\rho_i)$.
Then, by Lemma~\ref{C^2-estimate} we know that for each compact subset
$K \subset Z_\delta$ there is a constant $C>0$ with 
\[
\Vert \phi_i\Vert_{C^2(K)} \leq C \Vert \phi_i\Vert_{L^2(Z_\delta, g)}.
\]
Thus, $\Vert \phi_i\Vert_{C^2(K)}\leq C$. By Ascoli's Theorem, Lemma
\ref{lemma_Ascoli}, we know that $\phi_i$ then 
converges strongly in $C^1(K)$ to a spinor $\phi$. Since $Z_\delta$
tends to $M_\infty\setminus ( S\times \{0\})$ as $\delta \to 0$ a
diagonal subsequence argument tells us that 
$\phi \in C^1_{\rm loc}(M_\infty \setminus (S \times \{0\}))$ and
$D^g \phi = 0$ on $M_\infty \setminus (S \times \{0\})$.
From $\Vert \phi_i\Vert_{L^2(Z_{\delta},g)} \leq 1$ the spinors
$\phi_i$ converge weakly in $L^2$, the limit has to be the same
spinor $\phi$. Thus $\Vert \phi \Vert_{L^2(Z_\delta,g)} 
\leq \liminf \Vert \phi_i\Vert_{L^2(Z_{\delta},g)}\leq 1$ and 
$\Vert \phi \Vert_{L^2(M_\infty, g)}\leq 1$.
Now Lemma~\ref{removal_singularities} on removal of singularities
tells us that $D^g\phi = 0$ weakly on $(M_\infty, g)$.

It remains to show that $\phi$ is not identically zero. In the same
way as in the proof of Proposition~\ref{first_approx_prop} one shows
that $\phi_i \to \phi$ on compact subsets of $(M_\infty,g)$ and
\[
\int_{M' \cup (\partial M\times[0,1])} |\phi_i|^2 \, dv^{g'_{\rho_i}} \geq c
\]
for a positive constant $c$. This means that the $\phi_i$ cannot
escape to infinity. Assuming that $\phi=0$ we get
\[ \begin{split}
c &\leq 
\int_{M' \cup (\partial M\times[0,1])} |\phi_i|^2 \, dv^{g'_{\rho_i}} \\
&= 
\int_{(M'\setminus U'(\rho_i)) \cup (\partial M\times[0,1])} 
|\phi_i|^2 \, dv^{g'_{\rho_i}} 
+ \int_{U'(\rho_i)} |\phi_i|^2 \, dv^{g'_{\rho_i}} \\
&= 
\underbrace{
\int_{(M\setminus U_{S_{\rho_i}}(\rho_i)) \cup 
(\partial M \times[0,1])} |\phi_i|^2 \, dv^{g}
}_{\to 0} 
+ \int_{U'(\rho_i)} |\phi_i|^2 \, dv^{g'_{\rho_i}} .
\end{split} \]

Hence, we still have to rule out that $\phi_i$ concentrates in
the limit only in the attached neck. This follows immediately 
from Lemma~\ref{fundamental_estimate_first_n}. For 
$\rho_i<s<2s<R_{\rm max}/2$ we have $U'(2s) \setminus U'(s) = 
U_{S_{\rho_i}}(2s) \setminus U_{S_{\rho_i}}(s)\subset M$ and with
Lemma~\ref{fundamental_estimate_first_n} we get that 
\[
\int_{U_{S_{\rho_i}}(2s)\setminus U_{S_{\rho_i}}(s)} |\phi_i|^2 dv^g 
\geq 
\frac{1}{8} \int_{U'(s)} |\phi_i|^2 dv^{g'_{\rho_i}}\geq c
\] 
which contradicts that $\phi_i\to 0$ on compact subsets of $M_\infty$.

Thus, the harmonic spinors $\phi_i$ converge to a non-zero harmonic
$L^2$-spinor on $(M_\infty, g)$ as $i \to \infty$ which gives a 
contradiction since there are no such spinors for the metric $g$.
\end{proof}

\subsection{Step 3: Approximating with a product metric again}

We have now performed the first surgery, and we fix a metric 
$g' \definedas g'_{\rho_0}$ on $M'$ with the properties we need. That
is $g'\in \Rinv(M')$ and the metric is unchanged except near the
surgery sphere so the product structure from Step 1 in a
neighbourhood of $S\times [0,\infty]$ is preserved. The radius of this
neighbourhood will be again denoted by $R_{\rm max}$.

The surgery in the previous step consisted of removing a neighbourhood
$S^k \times B^{n-k+1}$ (this was $\{C\}$) of the surgery sphere, where
the radius of the ball $B^{n-k+1}$ is small. The boundary of the
resulting manifold is diffeomorphic to $S^k \times S^{n-k}$, and the
surgery is completed by attaching $B^{k+1} \times S^{n-k}$ (which we
called $\{C'\}$). 

We now define a submanifold $B \simeq B^{k+1}$ of $M'$, in
$M'_{\infty}$ we have $B \simeq \mR^{k+1} \simeq 
B^{k+1} \cup S^k \times [0,\infty)$. In $\{A\}$ and $\{B\}$
introduced in the previous subsection we set
$B \definedas S \times [-3\rho/2, \infty)$ (with respect to the
cylindrical structure), in part $\{C'\}$ we set $B \definedas 
B^{k+1} \times \{ p \}$ where $p \in S^{n-k}$ is chosen so that $B$ is
a smooth connected submanifold. The position of $B$ in 
$(M'_{\infty}, g')$ is illustrated in the left of Figure
\ref{figure_step3a}.

Let $i': B^{k+1} \to M'$ be the corresponding embedding with
$i'(B^{k+1}) = B$. The submanifold $B$ has a natural trivialization of
its normal bundle. 

In this section we will show that the metric $g'$
can be deformed to have a product structure in an arbitrarily small
neighborhood of $B$. In the subset 
$U_S( R_{\rm max} ) \times (0, \infty)$ of the cylindrical end we
already have 
\[
g' = \partial g + dt^2 = h + \xi^{n-k} + dt^2
\]
We will extend this product structure to a neighborhood of all of
$B$. 

Let $\pi^{\nu'}: \nu'\to B$ be the normal bundle of $B$ in $(M',g')$
and assume that a trivialization of $\nu'$ is given through a vector
bundle map $\iota': B^{k+1} \times \mR^{n-k} \to \nu'$ such that
$(\pi^{\nu'} \circ \iota')(p,0)=i'(p)$ for $p\in B^{k+1}$. Further
we assume that $\iota'$ is fiberwise an isometry when the fibers
$\mR^{n-k}$ of $B^{k+1}\times \mR^{n-k}$ are given the standard
metric, and the fibers of $\nu'$ have the metric induced by $g'$. For
sufficiently small $R$ we get an embedding $f' \definedas 
\exp^{\nu'} \circ \iota': B^{k+1} \times B^{n-k}(R) \to M'$. 
We define an open neighborhood of $B$ by 
\[
U_B(R) \definedas
(\exp^{\nu'}\circ\iota')(B^{k+1}\times B^{n-k}(R))
\]
for $R$ small enough. Let $h'$ denote the pullback by $i'$ to
$B^{k+1}$ of the restriction of $g'$ to the tangent bundle of $B$, 
and let $r(x)$ be the distance from the point $x \in M'$ to $B$. 
Note that in the cylindrical end $\partial M'\times [-t_1, \infty)$ 
we have $h' = h + dt^2$ and $r$ coincides with the previous definition.

\begin{prop} \label{Second_approx_prop}
The metric $g' \in \Rinv(M')$ can be arbitrarily closely approximated
by metrics $\overline{g}'_{\delta} \in \Rinv(M')$ which have the form
\[
\overline{g}'_{\delta} = h' + \xi^{n-k}
\]
on 
$U_B(\delta)$ and 
\[
\overline{g}'_{\delta} = g'
\]
outside $U_B(2\delta)$ and on the cylindrical end 
$\partial M' \times [0,\infty)$.
\end{prop}

We now define the metrics $\overline{g}'_{\delta}$ and then prove that
they have the required properties. Let $\chi$ be the cut-off function
introduced in Subsection~\ref{step1} and set 
$\eta(x) \definedas \chi(r(x) /\delta )$ where $\delta > 0$ is a
small parameter. We define
\[
\overline{g}'_{\delta} 
\definedas
\eta ({f'}^{-1})^* ( h' + \xi^{n-k} ) + (1 - \eta) g'
\]

To the right in Figure~\ref{figure_step3a} we have 
$(M'_{\infty}, \overline{g}'_{\delta})$ with the product region shaded.

\begin{figure}[h]
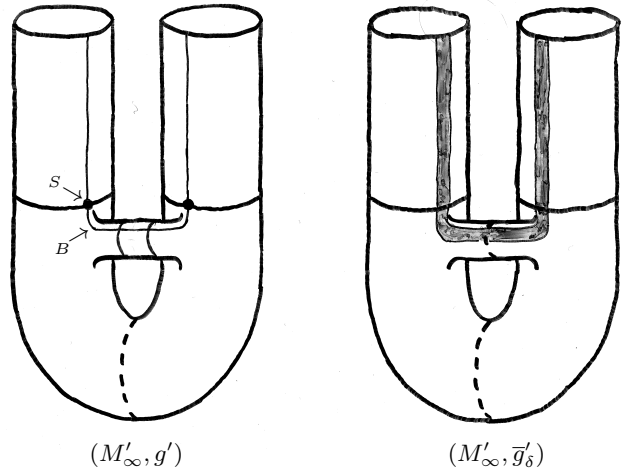

\begin{lpic}[]{FIGS/STEP3_A(0.33)}
\lbl[]{75,10;{\small $(M'_{\infty}, g')$}}
\lbl[]{220,10;{\small $(M'_{\infty}, \overline{g}'_{\delta})$}}
\lbl[]{45,93;{\tiny $B$}}
\lbl[]{52,98,30;{\small $\rightarrow$}}
\lbl[]{42,120;{\tiny $S$}}
\lbl[]{49,115,-30;{\small $\rightarrow$}}
\end{lpic}
\caption{Second approximation with product metric.}
\label{figure_step3a}
\end{figure}

\begin{proof} 
We need to prove $\overline{g}'_{\delta} \in \Rinv(M')$ for small
$\delta$, the other properties are clear. Again we argue by
contradiction and assume that there is a sequence $\delta_i \to 0$
such that $\overline{g}'_{\delta_i} \notin \Rinv(M')$ 

Since $\partial \overline{g}'_{\delta} = \partial g' = \partial g$ is
independent of $\delta$ we have uniform gaps around zero in the
essential spectrum of $D^{\overline{g}'_{\delta}}$, therefore the
assumption that $\overline{g}'_{\delta_i} \notin \Rinv(M')$ implies
the existence of harmonic spinors $\phi'_i$ on 
$(M'_{\infty}, \overline{g}'_{\delta_i})$ 
with $D^{\overline{g}'_{\delta_i}} \phi'_i = 0$ and 
$\int_{M'_\infty} |\phi'_i|^2 \, dv^{\overline{g}'_{\delta_i}} = 1$. 

The proof goes on exactly as the proof of Proposition
\ref{first_approx_prop}. We note that $\overline{g}'_{\delta_i} = g'$ 
on $M'_\infty\setminus U_B(2\delta)$. We fix $\gamma$ small
enough. Then for all $i$ with $2\delta_i < \gamma$ and all compact
subsets $K \subset M'_\infty \setminus U_B(\gamma) \subset 
M'_\infty \setminus U_B(2\delta_i)$ we get with
Lemma~\ref{C^2-estimate} that 
\[
\Vert\phi'_i\Vert_{C^2(K)} 
\leq 
C\Vert \phi'_i 
\Vert_{L^2(M'_\infty \setminus U_{B}(\gamma),g')}
\leq C 
\] 
where $C$ is a 
constant only depending on $(K, M'_\infty \setminus U_{B}(\gamma),g')$.
From the Theorem of Ascoli, Lemma~\ref{lemma_Ascoli}, we
obtain that $\phi'_i\to \phi'$ strongly in $C^1(K)$ and $D^{g'}\phi'=0$
weakly on each $K$. Moreover, $\phi'_i \to \phi'$ weakly in
$L^2(M'_\infty \setminus U_{B}(\gamma), g')$ and 
$\Vert \phi' \Vert_{L^2(M'_\infty\setminus U_{B}(\gamma), g')} \leq 1$. 
Thus, if $\gamma \to 0$ we obtain that $D^{g'}\phi'=0$ weakly on 
$M'_\infty \setminus B$ and $\phi' \in L^2(M'_\infty,g')$. From Lemma
\ref{removal_singularities} we then have $D^{g'}\phi'=0$ weakly on
$M'_\infty$. And again it remains to show that $\phi'$ does not vanish
identically. This is done exactly as in Proposition
\ref{first_approx_prop} using part (3) of Proposition
\ref{propbdrygap}. 
\end{proof}

After this step we replace $g'$ by $\overline{g}'_{\delta'_0}$ for
some $\delta'_0$ sufficiently small and define $R'_{\rm max}
\definedas \delta'_0$.

\subsection{Step 4: Second surgery}

In this section, we perform surgery (or ``half-surgery'') on 
$B$ in $(M',g')$ to produce $(M'', g''_{\rho})$. Here $\rho > 0$ is 
again a parameter which will be adjusted later. The aim is to 
replace a neighbourhood of $B$ which is diffeomorphic to 
$B^{k+1} \times B^{n-k}$ (see $\{F\}$ below) by 
$B_-^{k+2} \times S^{n-k-1}$ (see $\{ F' \}$ below). 

On $U_B(R'_{\rm max})$ the metric $g'$ has the product form 
\[
g' = h' + \xi^{n-k} = h' + dr^2 + r^2 \sigma^{n-k-1},
\] 
and in the cylindrical end where $h' = h + dt^2$ we have
\[
g' = h + dt^2 + \xi^{n-k} = h + dr^2 + r^2 \sigma^{n-k-1} + dt^2.
\] 

We divide $M'$ into three pieces, see Figure~\ref{figure_step4a},
\begin{itemize}
\item[$\{D\}$] 
$M' \setminus U_{B}(R'_{\rm max}/2)$,
\item[$\{E\}$] 
$U_{B}(R'_{\rm max}/2) \setminus U_{B}(\rho /2)
\simeq
B^{k+1} \times (\rho/2, R'_{\rm max}/2) \times S^{n-k-1} $,
\item[$\{F\}$]
$U_{B}(\rho /2)
\simeq
B^{k+1} \times B^{n-k} (\rho/2)$,
\end{itemize}
and we divide the cylindrical end of $M'_{\infty}$ in corresponding
pieces, 
\[
\partial M' \times [0,\infty) 
= \{\overline{D}\} + \{\overline{E}\} + \{\overline{F}\},
\]
where 
\[
\{\overline{D}\} = \{\partial D\} \times [0,\infty), \quad
\{\overline{E}\} = \{\partial E\} \times [0,\infty), \quad
\{\overline{F}\} = \{\partial F\} \times [0,\infty) 
\]
come from a decomposition of the boundary $\partial M' = \partial M$
into three pieces
\begin{itemize}
\item[$\{\partial D\}$] 
$\partial M' \setminus U_{S}(R'_{\rm max}/2)$,
\item[$\{\partial E\}$]
$U_{S}(R'_{\rm max}/2) \setminus U_{S}(\rho /2)
\simeq
S^{k} \times (\rho/2, R'_{\rm max}/2) \times S^{n-k-1} $,
\item[$\{\partial F\}$]
$U_{S}(\rho /2) \simeq S^k \times B^{n-k} (\rho/2) $.
\end{itemize}
Finally, we set 
\[
\{D_{\infty}\} = \{D\} + \{\overline{D}\}, \quad
\{E_{\infty}\} = \{E\} + \{\overline{E}\}, \quad
\{F_{\infty}\} = \{F\} + \{\overline{F}\}, 
\]
so that $M'_{\infty} = \{D_{\infty}\} + \{E_{\infty}\}
+ \{F_{\infty}\}$. 

\begin{figure}[h]
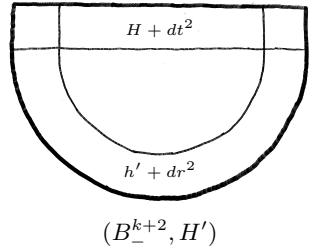

\begin{lpic}[]{FIGS/STEP4_B(0.30)}
\lbl[]{105,15;{\small $(B^{k+2}_-, H')$}}
\lbl[]{105,43;{\tiny $h' + dr^2$}}
\lbl[]{105,106;{\tiny $H + dt^2$}}
\end{lpic}
\caption{The metric $H'$.}
\label{figure_step4b}
\end{figure}

Let $B^{k+2}_-$ denote the lower half of the $(k+2)$-dimensional
disk. Let $H'$ be a metric on $B^{k+2}_-$ which is equal to $H + dt^2$
near the horizontal part of the boundary (For the definition of $H$
see Step 2.) and equal to $h' + dr^2$ near the hemisphere part of the 
boundary, see Figure~\ref{figure_step4b}. Near the corners these 
metrics coincide as $h + dr^2 + dt^2$.

The manifold $M''$ after surgery is obtained by replacing
$\{F\}$ by 
\begin{itemize}
\item[$\{F'\}$]
$B^{k+2}_- \times S^{n-k-1} $
\end{itemize}
and $\{\overline{F}\}$ by $\{\overline{F}'\} 
= \{\partial F'\} \times [0,\infty)$, where
\begin{itemize}
\item[$\{\partial F'\}$]
$B^{k+1} \times S^{n-k-1}$.
\end{itemize}
We define metrics $g''_{\rho}$ on $M''$ by 
\begin{itemize}
\item[$\{D\}$] 
$g''_{\rho} \definedas g'$
\item[$\{E\}$] 
$g''_{\rho} \definedas
h + dr^2 + \alpha_{\rho}(r)^2 \sigma^{n-k} $
\item[$\{F'\}$]
$g''_{\rho} \definedas
H' + (2\rho/3)^2 \sigma^{n-k-1} $
\end{itemize}
($\alpha_\rho$ is as defined in Figure~\ref{figure_alpharho} when 
replacing $R_{\text{max}}$ by $R_{\text{max}}'$)
and on the cylindrical end $g''_{\rho} 
\definedas \partial g''_{\rho} + dt^2$ where
\begin{itemize}
\item[$\{\partial D\}$] 
$\partial g''_{\rho} \definedas \partial g'$
\item[$\{\partial E\}$]
$\partial g''_{\rho} \definedas
h + dr^2 + \alpha_{\rho}(r)^2 \sigma^{n-k-1}$ 
\item[$\{\partial F'\}$]
$\partial g''_{\rho} \definedas
H + (2\rho/3)^2 \sigma^{n-k-1}$
\end{itemize}

In Figure~\ref{figure_step4a} we have $(M'_{\infty}, g')$ before
surgery to the left and $(M''_{\infty}, g''_{\rho})$ after surgery to
the right. Note that the boundary manifold 
$(\partial M'', \partial g''_{\rho})$ is the result of surgery on
$\partial M$ along the embedding $f : S^k \times B^{n-k} \to \partial M$.

\begin{figure}[h]
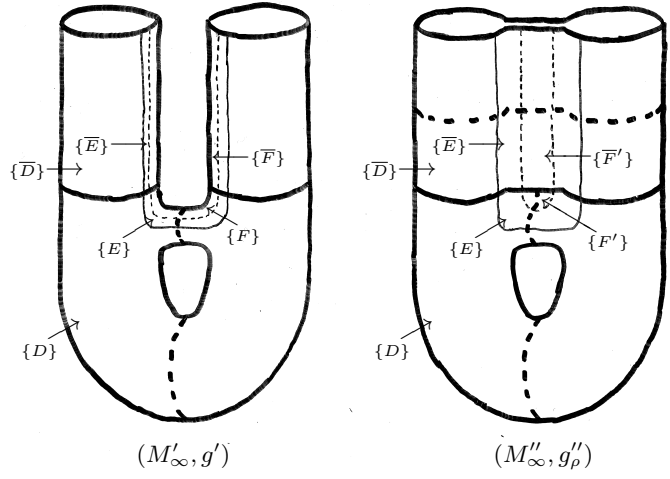

\begin{lpic}[]{FIGS/STEP4_A(0.33)}
\lbl[]{75,10;{\small $(M'_{\infty}, g')$}}
\lbl[]{220,10;{\small $(M''_{\infty}, g''_{\rho})$}}

\lbl[]{17,52;{\tiny $\{D\}$}}
\lbl[]{27,60,30;{\small $\longrightarrow$}}
\lbl[]{12,125;{\tiny $\{\overline{D}\}$}}
\lbl[]{27,125;{\small $\longrightarrow$}}
\lbl[]{47,93;{\tiny $\{E\}$}}
\lbl[]{57,101,30;{\small $\longrightarrow$}}
\lbl[]{38,135;{\tiny $\{\overline{E}\}$}}
\lbl[]{53,135;{\small $\longrightarrow$}}
\lbl[]{99,98;{\tiny $\{ F \}$}}
\lbl[]{91,107,150;{\small $\longrightarrow$}}
\lbl[]{109,130;{\tiny $\{\overline{F}\}$}}
\lbl[]{94,130;{\small $\longleftarrow$}}

\lbl[]{159,52;{\tiny $\{D\}$}}
\lbl[]{169,60,30;{\small $\longrightarrow$}}
\lbl[]{154,125;{\tiny $\{\overline{D}\}$}}
\lbl[]{169,125;{\small $\longrightarrow$}}
\lbl[]{189,93;{\tiny $\{E\}$}}
\lbl[]{199,101,30;{\small $\longrightarrow$}}
\lbl[]{182,135;{\tiny $\{\overline{E}\}$}}
\lbl[]{197,135;{\small $\longrightarrow$}}
\lbl[]{248,130;{\tiny $\{ \overline{F}' \}$}}
\lbl[]{230,134;{\small $\xleftarrow{\phantom{XX}}$}}
\lbl[]{246,96;{\tiny $\{ F' \}$}}
\lbl[]{229,104,150;{\small $\xrightarrow{\phantom{XXX}}$}}
\end{lpic}
\caption{Second surgery.}
\label{figure_step4a}
\end{figure}

For $R>\frac{\rho}{2}$ we define $U''(R) \subset M''_\infty$ by 
$M''_\infty \setminus U''(R) = M'_\infty \setminus U_B(R)$. Note that
on $M''_\infty \setminus U''(R)$ we have $g''_\rho = g'$. Further, we 
define the subset $\partial U''(R) \subset \partial M''$ by 
$\partial M'' \setminus \partial U''(R) = 
\partial M \setminus U_S(R)$ for $R > \frac{\rho}{2}$. Note that 
$\partial g_\rho'' = \partial g$ on
$\partial M'' \setminus \partial U''(R)$.

\begin{prop} \label{second_surgery_prop}
$g''_\rho \in \Rinv(M'')$ for all sufficiently small $\rho$.
\end{prop}

Before proving the Proposition we need to show that the boundary
metrics have a uniform spectral gap. For this we need the following
Lemma, similar to \cite[Proposition~3.5]{ammann_dahl_humbert_09} and
Lemma~\ref{fundamental_estimate_first_n}.

\begin{lemma} \label{fundamental_estimate_first}
Choose $\rho$ and $s$ so that $\rho < s < 2s < R'_{\rm max}/2$ and assume
that $\psi_\rho$ are spinors on $(\partial M'', \partial g''_\rho)$
satisfying 
\[
D^{\partial g''_\rho} \psi_{\rho}
= \lambda_{\rho} \psi_{\rho}
\]
where $32 \lambda_{\rho}^2 s^2 < 1/2$. Then 
\[
\frac{1}{128} \int_{\partial U''(s)}
|\psi_{\rho}|^2 \, dv^{\partial g''_\rho}
\leq
\int_{\partial U''(2s) \setminus \partial U''(s)}
|\psi_{\rho} |^2 \, dv^{\partial g''_\rho}.
\]
\end{lemma}

\begin{proof}
We make the conformal change $\partial \widetilde{g}_{\rho} 
= \alpha_{\rho}^{-2} \partial g''_\rho$ (the function $\alpha_\rho$
was defined in Step $2$) and set 
\[
\widetilde{\psi}_{\rho} = \alpha_{\rho}^{\frac{n-1}{2}} 
\beta^{\partial g''_\rho}_{\partial \widetilde{g}_{\rho}}
\psi_{\rho}.
\]
We then have
\[
D^{\partial \widetilde{g}_{\rho}} \widetilde{\psi}_{\rho}
= \lambda_{\rho} \alpha_{\rho} \widetilde{\psi}_{\rho}.
\]
Choose $s$ so that 
\[
\rho < s < 2s < R'_{\rm max}/2
\]
and choose a cut-off function $\eta$ on $\partial M''$ with 
$\eta = 1$ on $\partial U''(s)$ and $\eta = 0$ on 
$\partial M'' \setminus \partial U''(2s)$. 
Since $d\eta$ is supported in 
$\partial U''(2s)\setminus \partial U''(s)\subset \{\partial E\}$ 
we may assume that 
\[ 
|d\eta|_{\partial g''_\rho} \leq 2/s
\]
which implies 
\[ 
|d\eta|^2_{\partial \widetilde{g}_{\rho}} 
= 
\alpha_{\rho}^2 |d\eta|^2_{\partial g''_\rho}
\leq 
4 \alpha_{\rho}^2 /s^2.
\]
We have
\[
D^{\partial \widetilde{g}_{\rho}} (\eta \widetilde{\psi}_{\rho})
= 
\grad^{\partial \widetilde{g}_{\rho}} \eta \cdot \widetilde{\psi}_{\rho} 
+ 
\eta \lambda_{\rho} \alpha_{\rho} \widetilde{\psi}_{\rho},
\]
so
\begin{equation} \label{Dgrho}
\begin{split}
|D^{\partial \widetilde{g}_{\rho}} (\eta \widetilde{\psi}_{\rho})|^2
&\leq 
2 |\grad^{\partial \widetilde{g}_{\rho}} \eta \cdot
\widetilde{\psi}_{\rho}|^2 
+
2 \lambda_{\rho}^2 \alpha_{\rho}^2 |\eta \widetilde{\psi}_{\rho}|^2 \\
&\leq 
\frac{8\alpha_{\rho}^2}{s^2} |\widetilde{\psi}_{\rho} |^2
+ 
2 \lambda_{\rho}^2 \alpha_{\rho}^2 |\eta \widetilde{\psi}_{\rho}|^2
\end{split}
\end{equation}
where the first term is supported in 
$\partial U''(2s) \setminus \partial U''(s)$.
Since $\partial \widetilde{g}_{\rho} = 
\sigma^{n-k-1} + \alpha_{\rho}^{-2} H$ on $\partial U''(2s)$ we have
a lower spectral bound, see \cite[Lemma~2.5]{ammann_dahl_humbert_09},
\begin{equation} \label{specbound}
\begin{split}
\int_{\partial U''(2s)} 
|D^{\partial \widetilde{g}_{\rho}} (\eta \widetilde{\psi}_{\rho})|^2
\, dv^{\partial \widetilde{g}_{\rho}}
&\geq 
\frac{(n-k-1)^2}{4}
\int_{\partial U''(2s)} 
|\eta \widetilde{\psi}_{\rho}|^2 
\, dv^{\partial \widetilde{g}_{\rho}} \\ 
&\geq 
\frac{1}{4}
\int_{\partial U''(2s)} 
|\eta \widetilde{\psi}_{\rho}|^2 
\, dv^{\partial \widetilde{g}_{\rho}} .
\end{split}
\end{equation}
Using \eqref{Dgrho} we get for the left-hand side, 
\[ \begin{split}
&\int_{\partial U''(2s)} 
|D^{\partial \widetilde{g}_{\rho}} (\eta \widetilde{\psi}_{\rho})|^2
\, dv^{\partial \widetilde{g}_{\rho}} \\
&\qquad \leq
\frac{8}{s^2} 
\int_{\partial U''(2s) \setminus \partial U''(s)}
\alpha_{\rho}^2 |\widetilde{\psi}_{\rho} |^2 
\, dv^{\partial \widetilde{g}_{\rho}}
+
2 \lambda_{\rho}^2 \int_{\partial U''(2s)}
\alpha_{\rho}^2 |\eta \widetilde{\psi}_{\rho}|^2 
\, dv^{\partial \widetilde{g}_{\rho}} \\
&\qquad =
\frac{8}{s^2} \int_{\partial U''(2s) \setminus \partial U''(s)}
 \alpha_{\rho} |\psi_{\rho} |^2 
\, dv^{\partial g''_\rho}
+
2 \lambda_{\rho}^2 \int_{\partial U''(2s)}
\alpha_{\rho}^2 |\eta \widetilde{\psi}_{\rho}|^2 
\, dv^{\partial \widetilde{g}_{\rho}} \\
&\qquad \leq
\frac{16}{s} 
\int_{\partial U''(2s) \setminus \partial U''(s)}
|\psi_{\rho} |^2 \, dv^{\partial g''_\rho}
+
8 \lambda_{\rho}^2 s^2 \int_{\partial U''(2s)}
|\eta \widetilde{\psi}_{\rho}|^2 
\, dv^{\partial \widetilde{g}_{\rho}}
\end{split} \]
where we used that $\alpha_{\rho} \leq 2s$ in the final step. Inserted
in \eqref{specbound} we get
\begin{equation} \label{bound2}
\frac{16}{s} 
\int_{\partial U''(2s) \setminus \partial U''(s)}
|\psi_{\rho} |^2 \, dv^{\partial g''_\rho}
\geq
\left(\frac{1}{4} - 8 \lambda_{\rho}^2 s^2\right)
\int_{\partial U''(2s)} |\eta \widetilde{\psi}_{\rho}|^2 
\, dv^{\partial \widetilde{g}_{\rho}}.
\end{equation}
Here we have for the right-hand side,
\[ \begin{split}
\int_{\partial U''(2s)} 
|\eta \widetilde{\psi}_{\rho}|^2 
\, dv^{\partial \widetilde{g}_{\rho}}
&\geq
\int_{\partial U''(s)} 
|\widetilde{\psi}_{\rho}|^2 
\, dv^{\partial \widetilde{g}_{\rho}} \\
&=
\int_{\partial U''(s)}
\alpha_{\rho}^{-1}
|\psi_{\rho}|^2 
\, dv^{\partial g''_\rho} \\
&\geq
\frac{1}{s} 
\int_{\partial U''(s)}
|\psi_{\rho}|^2 
\, dv^{\partial g''_\rho}
\end{split} \]
where we in the final step used that $\alpha_{\rho} \leq s$. Inserted
in \eqref{bound2} we get 
\[
\frac{16}{s} 
\int_{\partial U''(2s) \setminus \partial U''(s)}
|\psi_{\rho} |^2 \, dv^{\partial g''_\rho}
\geq
\left(\frac{1}{4} - 8 \lambda_{\rho}^2 s^2\right)
\frac{1}{s} 
\int_{\partial U''(s)}
|\psi_{\rho}|^2 
\, dv^{\partial g''_\rho} ,
\]
or
\[ 
\int_{\partial U''(2s) \setminus \partial U''(s)}
|\psi_{\rho} |^2 \, dv^{\partial g''_\rho}
\geq
\frac{1 - 32 \lambda_{\rho}^2 s^2}{64}
\int_{\partial U''(s)}
|\psi_{\rho}|^2 
\, dv^{\partial g''_\rho} ,
\]
from which the lemma follows. 
\end{proof}

We also need one more version of this estimate.

\begin{lemma} \label{fundamental_estimate_second}
Let $s$ be such that $\rho < s < 2s < R_{\rm max} /2$ and assume that 
$\psi_\rho$ are harmonic $L^2$-spinors on $(M''_\infty, g''_\rho)$,
that is $D^{g''_{\rho}} \psi_{\rho} = 0$ and 
$\int_{M'_\infty} |\psi_\rho|^2 dv^{g''_{\rho}} < \infty$. Then 
\[
\frac{1}{8} 
\int_{U''(s)} |\psi_{\rho}|^2 \, dv^{g''_{\rho}}
\leq
\int_{U''(2s) \setminus U''(s)} 
| \psi_{\rho} |^2 \, dv^{g''_{\rho}}.
\]
\end{lemma}

\begin{proof}
The proof is similar to the ones for Lemmas
\ref{fundamental_estimate_first_n} and~\ref{fundamental_estimate_first}.
%
%
\end{proof}

We can now show that the boundary metrics $\partial g''_\rho$
have a uniform spectral gap. 

\begin{lemma} \label{bdry_surgery_prop}
There is a $\Lambda > 0$ such that
$\spec D^{\partial g''_\rho} \cap [-\Lambda, \Lambda] = \emptyset$
for all sufficiently small $\rho$.
\end{lemma}
\begin{proof}
For a contradiction assume that there is a sequence $\rho_i \to 0$
such that there are eigenspinors 
$\phi_i\in L^2(\partial M'', \partial g_{\rho_i}'')$ with
$D^{\partial g_{\rho_i}''} \phi_i = \lambda_i \phi_i$ and 
$\lambda_i \to 0$. We normalize the eigenspinors by 
$\int_{\partial M''} |\phi_i|^2 \, dv^{\partial g''_{\rho_i}} = 1$. 

Fix $\delta > 0$. Then with $\partial g''_{\rho_i}=\partial g$ on
$\partial M''\setminus \partial U''(\delta)$ for all $i$ with $\rho_i
< \delta$ and from Lemma~\ref{C^2-estimate} we get that for those $i$ 
$\phi_i$ is uniformly bounded in $C^2(\partial M'' \setminus \partial
U'' (\delta),g)$. Due to Ascoli's theorem, Lemma~\ref{lemma_Ascoli},
we get that $\phi_i \to \phi$ strongly in 
$C^1(\partial M'' \setminus \partial U'' (\delta))$
and $D^{\partial g} \phi = 0$ weakly on 
$\partial M \setminus U_S(\delta)$. Letting $\delta $ tend to zero and
taking a diagonal sequence we find that $D^{\partial g} \phi = 0$
weakly on $\partial M \setminus S$. Since 
$\Vert \phi_i \Vert_{L^2(\partial M \setminus U_S(\rho_i))} \leq 1$, 
we get $\phi \in L^2(\partial M \setminus S)$. Using the result on
removal of singularities on closed manifolds in 
\cite[Lemma~2.4]{ammann_dahl_humbert_09} we see that 
$D^{\partial g} \phi = 0$ holds weakly on $\partial M$. 

It remains to show that $\phi$ does not vanish identically. For a
fixed $s < R'_{\rm max}/4$ Lemma~\ref{fundamental_estimate_first} 
gives
\[ \begin{split}
\frac{1}{128} \int_{\partial U''(s)} |\phi_i|^2
\, dv^{\partial g''_{\rho_i}}
&\leq 
\int_{\partial U''(2s) \setminus \partial U''(s)} |\phi_i|^2
\, dv^{\partial g''_{\rho_i}} \\
&\leq 
\int_{\partial M'' \setminus \partial U''(s)} 
|\phi_i|^2 \, dv^{\partial g''_{\rho_i}} 
\end{split} \]
for all $i$ with $\rho_i<s$ and
$\lambda_{\rho_i}<\frac{4}{\sqrt{s}}$. Therefore, we get 
\[
\frac{1}{128} \int_{\partial M''} |\phi_i|^2
\, dv^{\partial g''_{\rho_i}}
\leq 
(1 + \frac{1}{128})
\int_{\partial M'' \setminus \partial U''(s)}
|\phi_i|^2 \, dv^{\partial g''_{\rho_i}} ,
\]
and
\[
1 \leq 
129\int_{\partial M \setminus U_S(s)} |\phi_i|^2
\, dv^{\partial g}.
\]
Since $\partial M\setminus U_S(s)$ is compact, $\phi$ cannot vanish
identically. Thus, $\phi$ is a harmonic spinor on 
$(\partial M, \partial g)$ which gives the required contradiction. 
\end{proof}

Finally, we are ready to prove Proposition~\ref{second_surgery_prop}.

\begin{proof}[Proof of Proposition~\ref{second_surgery_prop}]
From Lemma~\ref{bdry_surgery_prop} and Proposition
\ref{propbdrygap} we know that $D^{g''_\rho}$ 
has a uniform gap in the essential spectrum for all small $\rho$. 
We argue by contradiction and assume that there are $L^2$-harmonic
spinors $\phi_i$ for $g''_{\rho_i}$ as $\rho_i \to 0$. We normalize
the harmonic spinors by 
$\int_{M''_\infty} |\phi_i|^2 \, dv^{g''_{\rho_i}} = 1$. The goal is to
prove that these converge to an $L^2$-harmonic spinor on 
$(M_\infty' \setminus B,g')$, which then gives an $L^2$-harmonic
spinor on $(M_\infty',g')$ and, thus, a contradiction. 

The next step is similar to the proof of~\ref{first_approx_prop}.
Fix $\delta > 0$ small enough. Note that for all $i$ with
$\rho_i<\delta$ we have $(M_\infty''\setminus U''(\delta),
g''_{\rho_i})=(M_\infty'\setminus U_B(\delta), g')$. By Lemma
\ref{C^2-estimate} we obtain that $\phi_i$ is 
uniformly bounded in $C^2(K)$ for any compact subset
$K \subset M'_\infty \setminus U_B(\delta)$. From Ascoli's Theorem,
Lemma~\ref{lemma_Ascoli}, we get $\phi_i \to \phi$ strongly in 
$C^1(K)$ and $D^{g'}\phi=0$ weakly on each $K$. Thus, 
$\phi \in C^1(M'_\infty\setminus U_B(\delta))$. Hence, if $\delta\to
0$, we get a spinor $\phi \in C^1_{\rm loc}(M_\infty' \setminus B)$
with $D^{g'} \phi = 0$ weakly on $M'_\infty \setminus B$. Using 
Lemma~\ref{removal_singularities} we see that $D^{g'} \phi = 0$ on
$M'_\infty$. 

It remains again to show that $\phi$ does not vanish identically. 
For a fixed $\delta \in (0, R'_{\rm{max}})$ we get from 
Lemma~\ref{fundamental_estimate_second} that
\[ \begin{split}
\frac{1}{8} 
\int_{U''(\delta)} |\phi_i|^2 \, dv^{g''_{\rho_i}}
&\leq 
\int_{U''(2\delta) \setminus U''(\delta)} 
| \phi_i |^2 \, dv^{g''_{\rho_i}}\\
&\leq 
\int_{M''_\infty \setminus U''(\delta)} 
| \phi_i |^2 \, dv^{g''_{\rho_i}}\\
\end{split} \]
for all $i$ with $\rho_i < \delta$. It follows that 
\[
\frac{1}{8} 
= 
\frac{1}{8} \int_{M''_\infty} |\phi_i|^2 \, dv^{g''_{\rho_i}} 
\leq
\frac{9}{8} \int_{M''_\infty \setminus U''(\delta)} 
| \phi_i |^2 \, dv^{g''_{\rho_i}}, 
\]
or 
\[
\frac{1}{9}
\leq \int_{M'_\infty \setminus U_B(\delta)} | \phi_i |^2 \, dv^{g'}.
\]
Since we know from Proposition~\ref{propbdrygap} that each harmonic
spinor $\phi_i$ decays exponentially, this implies as in the proof 
of Proposition~\ref{first_approx_prop} that $\phi_i$ cannot converge 
to zero on compact subsets. Hence, $\phi$ cannot be identically zero. 
Thus, $\phi$ is a nontrivial $L^2$-harmonic spinor on $(M_\infty', g')$ 
which is a contradiction.
\end{proof}

\begin{proof}[Proof of Theorem~\ref{thm_handle_attachment}]
The Theorem now follows by choosing 
$g'' = g''_{\rho_0}$ with $\rho_0$ sufficiently small. 
\end{proof}

\section{Non-isotopic metrics with invertible Dirac operator}
\label{section_non-isotopic}

In this section we show that $\Rinv(M)$ has infinitely many components
if $\dim M = 3$. This extends previous results using surgery techniques
for positive scalar curvature \cite[Theorem~4]{carr_88}, 
\cite[Chapter~4, Theorem~7.7]{lawson_michelsohn_89}, and for invertible 
Dirac operator \cite[Theorem~3.3]{dahl_08}, where $\dim M = 4m-1$, 
$m \geq 2$. For the case of metrics with invertible Dirac operator the 
fact that $\Rinv(M)$ has infinitely many components follows more easily 
from the explicit examples of spectral flow constructed in 
\cite{hitchin_74}, \cite{baer_96} using families of Berger metrics.
The motivation for the argument given here is primarily to illustrate 
the use of surgery techniques to prove spectral results for the Dirac 
operator.

\begin{defn} \label{def_concordance} 
Let $M$ and $N$ be closed Riemannian spin manifolds with metrics
$g^0, g^1\in \Rinv(M)$ and $h\in \Rinv(N)$. 
\begin{enumerate}
\item 
$g^0$ and $g^1$ are called concordant if there exists a
metric $\overline{g}\in \Rinv([0,1]\times M)$ with
$\overline{g}|_{\{i\}\times M}=g^i$ for $i=0,1$. 
\item 
$g^0$ and $g^1$ are called isotopic if there exists a
smooth path of metrics $g_t$ in $\Rinv(M)$ ($t\in\mathbb{R}$) with
$g_t=g^0$ for $t\leq 0$ and $g_t=g^1$ for $t\geq 1$. 
\item 
$g^0$ and $h$ are called bordant if there is a manifold $W$ with 
a metric $g^W\in \Rinv(W)$ and $\partial (W, g^W) = 
(M,g^0) \sqcup (N^-, h)$ where $N^-$ denotes the manifold $N$ 
equipped with the reverse orientation.
\end{enumerate}
\end{defn}

Both isotopy and concordance are equivalence relations
\cite[Corollary~2.2]{dahl_08}. The corresponding sets of equivalence 
classes are denoted by $\conc \Rinv(M)$ for the concordance classes 
and by $\pi_0 \Rinv(M)$ for the isotopy classes. Isotopic metrics are 
concordant \cite[Corollary~2.1]{dahl_08}, this is the reason why
non-concordant metrics can be used to detect path components in
$\Rinv(M)$.

We will use the handle attachment result to construct non-concordant
metrics in $\Rinv(S^3)$---and the same for other 3-manifolds---from a 
handle decomposition of a 4-manifold with non-zero index.

\begin{lemma}
There are $4$-dimensional spin manifolds $Y^i$, $i \in \mZ$, with
boundary $\partial Y^i = S^3$, and metrics $g^{Y^i} \in \Rinv(Y^i)$
for which $\alpha(Y^i \cup_{S^3} (Y^j)^-) = c(i-j)$ where $c \neq 0$.
\end{lemma}

\begin{proof}
We let $Y^0$ be the 4-dimensional ball $B^4$ with a ``torpedo''
metric $g^{Y^0} \in \Rinv(Y^0)$ of positive scalar curvature, such
that $\partial g^{Y^0}$ is the standard round metric on $S^3$, see for
example \cite[Section~1.3]{walsh_11}. For positive $i$ we define the
manifolds $Y^i$ as the connected sum of $i$ copies of the ${\rm K3}$
surface with an open disc removed. For negative $i$ we set 
$Y^i \definedas (Y^{-i})^-$. Using the spin bordism invariance of 
$\alpha$ we have
\[
\alpha(Y^i \cup_{S^3} (Y^j)^-) 
= (i-j) \alpha({\rm K3}) 
= c(i-j)
\]
where $c \definedas \alpha({\rm K3}) \neq 0$. It remains to find
metrics $g^{Y^i} \in \Rinv(Y^i)$ for $i > 0$. 

From \cite[Corollary~6.3.19]{gompf_stipsicz_99} we know that there
exists a handle decomposition of the ${\rm K3}$ surface which starts
from the 4-dimensional ball $B^4$, then attaches a number of
$2$-handles $B^2\times B^2$, before finishing by attaching a $B^4$. 
This means that $Y^i$ can be obtained by attaching a number of 
$2$-handles to an initial $B^4$. Starting with the metric $g^{Y^0}$ 
on $B^4$ we apply Theorem~\ref{thm_handle_attachment} to extend it 
over the $2$-handles to a metric $g^{Y^i} \in \Rinv(Y^i)$.
\end{proof}

Let $h^i \in \Rinv(S^3)$ be defined by $h^i \definedas g^{Y^i}|_{S^3}$.

\begin{prop} \label{non_concordant_metrics} 
Suppose $M$ is a closed $3$-dimensional Riemannian spin manifold and 
$g \in \Rinv(M)$. Then there are metrics $g^i \in \Rinv(M)$, $i\in\mZ$,
such that $g^i$ is bordant to $g$ but $g^i$ is not concordant to $g^j$
for $i \ne j$. 
\end{prop}

\begin{proof}
By Theorem~\ref{thm_handle_attachment} there is for $i \in \mZ$
a metric $g^i$ on $M \# S^3 = M$ which is bordant to $g \sqcup h^i$ 
on $M \sqcup S^3$. The metric $h^i \in \Rinv(S^3)$ is bordant to zero
through the bordism $(Y^i,g^{Y^i})$, using 
\cite[Proposition~2.1]{dahl_08} we can attach this bordism to the
handle attachment bordism and conclude that $g^i$ is bordant to
$g$. Denote by $(W^i, g^{W^i})$ the bordism between $(M,g^i)$ and
$(M,g)$ we have now constructed. The manifold $W^i$ is diffeomorphic
to the boundary connected sum of $[0,1] \times M$ and $Y^i$.

For $i,j\in \mZ$ assume that $g^i, g^j \in \Rinv(M)$ are
concordant. We then find a metric with invertible Dirac operator on
the closed manifold $W^i \cup (W^j)^-$ obtained by attaching the
identical (but oppositely oriented) boundary components $(M,g)$ to
each other, and by attaching 
$(M,g^i)$ to $(M,g^j)$ using a concordance of the metrics. Then
$\alpha(W^i \cup (W^j)^-) = 0$. Further, $W^i \cup (W^j)^-$ is
diffeomorphic to the connected sum 
$(S^1 \times M) \# (Y^i \cup_{S^3} (Y^j)^-)$, so 
\[
0 = \alpha(W^i \cup (W^j)^-)
= \alpha(S^1 \times M) + \alpha(Y^i \cup_{S^3} (Y^j)^-)
= \alpha(Y^i \cup_{S^3} (Y^j)^-)
= c(i-j)
\]
and we conclude that $i=j$.
\end{proof}

\section{Concordance theory}
\label{section_concordance}

In this section we study the concordance classes of metrics with 
invertible Dirac operator on a manifold with boundary. Following 
closely the work by Stolz for positive scalar curvature we prove 
an existence and classification theorem, see \cite{stolz_95}, 
\cite{stolz_xx}. For previous work in the positive scalar curvature 
case see \cite{hajduk_91}, \cite{gajer_93}.

Let $M$ be a manifold with boundary, and let 
$h \in \Rinv(\partial M)$. We define ${\mathcal R}(M \rel h)$ as 
the set of Riemannian metrics $g$ on $M$ for which $\partial g = h$. 
Further we set $\Rinv(M \rel h) \definedas {\mathcal R}(M \rel h) 
\cap \Rinv(M)$.

By $\Ospin_n$ we denote the ordinary spin bordism group of 
dimension $n$. We also define 
\[ 
\Oinv_n
\definedas 
\{ (M,g) \mid M \text{ is a closed spin $n$-manifold},\ 
g \in \Rinv(M) \} / \sim,
\]
where the equivalence relation $\sim$ is defined by 
$(M_0,g_0) \sim (M_1,g_1)$ if there is a spin manifold $W$ with 
$\partial W = M_0 \sqcup M_1$ and a metric
$H \in \Rinv(W)$ such that $H|_{\partial W}= g_0 \sqcup g_1$ and 
all involved orientations and spin structures are compatible.

\subsection{Manifolds with corners}
\label{subsec_manifolds_with_corners}

To study concordances of metrics on manifolds with boundary it is
necessary to extend most of the theory and results obtained so far
to manifolds with corners. 

A manifold $M$ of dimension $n$ with corners of codimension $2$ is 
a smooth manifold with charts modelled on open sets in 
$\mR^{n-2} \times (-\infty, 0]^2$. Points with a neighbourhood 
diffeomorphic to a neighbourhood of the boundary of 
$\mR^{n-2} \times (-\infty, 0]^2$ constitute the boundary
$\partial M$ of $M$. Points with a neighbourhood 
diffeomorphic to a neighbourhood of the corner of 
$\mR^{n-2} \times (-\infty, 0]^2$ constitute the corner 
$\partial^2 M$ of $M$. We assume that the boundary itself 
constitutes an embedded submanifold with boundary in $M$.
We consider only Riemannian metrics $g$ on $M$ 
which have a product structure $g = \partial g + dt^2$ near 
$\partial M$ and a double product structure 
$g = \partial^2 g + dt_1^2 + dt_2^2$ near $\partial^2 M$.

As for manifolds with boundary we let $M_{\infty}$ be the manifold 
$M$ with half-infinite cylindrical ends attached, 
\[ 
(M_\infty,g) 
\definedas 
(M,g) 
\cup (\partial M \times [0,\infty), \partial g + dt^2 )
\cup (\partial^2 M \times [0,\infty)^2, \partial^2 g + dt_1^2 + dt_2^2).
\]

Now, $(M_\infty, g)$ is a complete Riemannian spin manifold. Thus, as 
in the case of manifolds with boundaries we can define the notion of 
invertibility of the Dirac operator and we have corresponding results 
for its spectrum.

We say that $(M,g)$ has invertible Dirac operator if the Dirac operator of 
$(M_\infty ,g)$ is invertible when it acts on $L^2$-sections of the spinor 
bundle.

The next proposition gives information about the spectral theory on those 
manifolds and is a version of Proposition \ref{propbdrygap} for manifolds 
with corners.

\begin{prop} \label{propbdrygap_corners}
Let $(M, g)$ be a Riemannian spin manifold with corners $X_i$. Let the 
boundary $\partial M$ be decomposed into finitely many manifolds with 
boundaries $N_i$ such that each boundary $\partial N_i$ is a corner 
$X_{j(i)}$. Assume that the Dirac operator on $(X_i, \partial^2 g)$ and 
the Dirac operator on $((N_i)_\infty, \partial g)$ are invertible. 
Moreover, let $M_T \definedas M 
\cup \left( \bigcup_i N_i \times [0,T] \right) 
\cup \left( \bigcup_i X_i\times [0,T]^2 \right)$ with the obvious 
identifications of the boundaries. Then the following holds.
\begin{enumerate}
\item \cite[Prop. 6.1]{mueller_96} 
The Dirac operator on $(M_\infty, g)$ is invertible. 
\item \cite[Prop. 2.19]{mueller_96}
There are constants $c,C >0$ such that for all harmonic spinors 
$\phi$ on $M_\infty$
\begin{equation} \label{harmonicdecaycorners}
\int_{M\setminus M_T} |\phi|^2 \, dv^{ g} \leq 
Ce^{-cT} \Vert \phi\Vert_{L^2(M_\infty)}^2
\end{equation}
for all $L^2$-harmonic spinors $\phi$ on $M_\infty$.
\item \cite[from the proof of Prop. 2.19]{mueller_96} 
Let $\Lambda>0$ be such that the Dirac operators on 
$((N_i)_\infty,\partial g)$ have a spectral gap on $(-\Lambda, \Lambda)$. 
Then in \eqref{harmonicdecaycorners} the constants can be chosen as
$c = \Lambda$ and $C=2$.
\end{enumerate}
\end{prop}

Bordisms of manifolds with boundary are naturally manifolds with 
corners. Such a bordism gives rise to a boundary bordism between 
the boundaries. For manifolds with boundary there are obvious 
extensions of the definitions of concordance and isotopy to 
$\Rinv(M \rel h)$. Note that the concordance relation for manifolds 
with boundary then assumes an invertible Dirac operator on a manifold 
with corners. 

Elementary constructions can be performed for metrics with invertible
Dirac operator. A product $M \times N$ with corners has invertible 
Dirac operator if at least one of the factors has. Attaching isometric 
boundary components by a sufficiently long attaching cylinder 
preserves invertibility of the Dirac operator, compare 
\cite[Proposition~2.1]{dahl_08}. Stretching an isotopy of metrics with 
invertible Dirac operator produces a concordance, compare
\cite[Proposition~2.3]{dahl_08}

For a smooth manifold $M$ with corner there is a procedure to round
the corner, producing a smooth manifold $\widetilde{M}$ with boundary.
Next we show that corners can be rounded while preserving invertibility 
of the Dirac operator. Let $\tau$ be a metric on a two-dimensional 
triangular domain $T$ which is a product near the boundary lines and a 
double product near the corners. 

Assume that $M$ is a manifold with corners and $g \in \Rinv(M)$. We
replace the corner piece $\partial^2 M \times [0,\infty)^2$ of 
$(M_\infty, g)$ by a part of 
$( (\partial^2 M \times T)_\infty, \partial^2 g + \tau)$, 
see Figure~\ref{figure_CONC_ROUNDING_CORNER}. 
\begin{figure}[h]
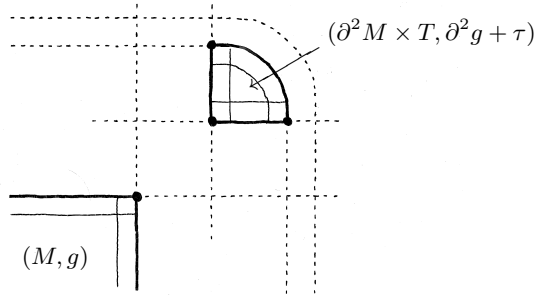

\begin{lpic}[]{FIGS/CONC_ROUNDING_CORNER(0.25)}
\lbl[]{70,50;{\small $(M,g)$}}
\lbl[]{190,155,30;{\small $\xleftarrow{\phantom{XXXX}}$}}
\lbl[]{270,170;{\small $(\partial^2 M \times T, \partial^2 g + \tau)$}}
\end{lpic}
\caption{Rounding the corner.}
\label{figure_CONC_ROUNDING_CORNER}
\end{figure}
If $M$ and $\partial^2 M \times T$ are sufficiently far apart we can 
use a cut-off function with sufficiently small gradient to conclude that 
the resulting manifold with boundary $(\widetilde{M}, \widetilde{g})$ 
has invertible Dirac operator, compare \cite[Proposition~2.1]{dahl_08}.

Next we extend Theorem~\ref{thm_handle_attachment} to manifolds with 
corners. 

\begin{thm} \label{thm_handle_attach_corners}
Let $(M,g)$ be a manifold with corners and $g\in \Rinv(M)$. Let $M''$ be 
obtained by a handle attachment outside a neighbourhood of the corners 
and of codimension at least two. Then for any given neighbourhood of the 
surgery sphere there is a metric $g''\in \Rinv(M'')$ such that $g''=g$ 
outside this neighbourhood. 
\end{thm}

\begin{proof} 
In principle the proof follows the proof of 
Theorem~\ref{thm_handle_attachment} since the handle attachment is done 
outside a neighbourhood of the corners. The steps explained in the 
strategy \ref{strategy} remain the same. But one has to make sure that 
all the auxiliary lemmas can be adapted to the new situation. Next, we 
will describe the required changes in those lemmas and in the proof.

Step 1: In Lemma \ref{lemmabdrygap} the boundary 
$(\partial M, \partial g)$ will now be itself a manifold with boundary. 
Thus, the statement is then just Proposition~\ref{propbdrygap}. The proof 
of Proposition~\ref{first_approx_prop} is done for corners analogously as 
before. But we now use Proposition~\ref{propbdrygap_corners} instead of 
Proposition~\ref{propbdrygap}. Moreover, the formulation of the 
Lemma~\ref{removal_singularities} for the removal of singularities has 
to be adapted to manifolds with corners. But its proof is exactly the 
same provided that $S$ is placed outside a neighbourhood of the corners.

Step 2 and 3 can be done in the same way using 
Proposition~\ref{propbdrygap_corners}. 

Step 4: The auxiliary Lemma \ref{bdry_surgery_prop} is now needed 
for manifolds with boundary which is exactly the result of 
Theorem~\ref{thm_handle_attachment}. The rest of this step is done 
analogously to the adaptations discussed before.
\end{proof}

\subsection{The $R^{\rm inv}_n$ groups and statement of the Theorem}
\label{Rngroups}

Following \cite[Definition 4.1]{stolz_xx} we define 
\[
R^{\rm inv}_n \definedas
\{
(M,h) \mid M \text{ is a spin $n$-manifold}, h \in \Rinv(\partial M)
\} / \sim,
\]
where $\partial M$ and $M$ are allowed to be empty, and $M$ is not
required to be connected. The equivalence relation $\sim$ is defined
by $(M_0, h_0) \sim (M_1, h_1)$ if 
\begin{itemize}
\item 
there is a spin manifold $V$ with 
$\partial V = \partial M_0 \sqcup \partial M_1$
and a metric $H \in \Rinv(V)$ such that 
$H|_{\partial V} = h_0 \sqcup h_1$,
\item 
there is a spin manifold $W$ with boundary 
$M_0 \cup_{\partial M_0} V \cup_{\partial M_1} M_1$,
\item 
the orientations and spin structures on all manifolds involved are
compatible in the obvious ways.
\end{itemize}
This is illustrated in Figure~\ref{figure_CONC_Rn_EQUIVALENCE}. 
The equivalence class of $(M, h)$ is denoted by $[M,h]$. 
\begin{figure}[h]
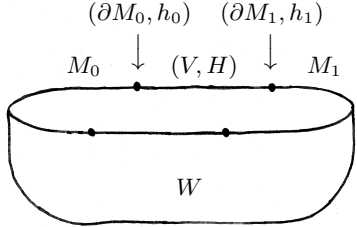

\begin{lpic}[]{FIGS/CONC_Rn_EQUIVALENCE(0.2)}
\lbl[]{80,150;{\small $M_0$}}
\lbl[]{240,150;{\small $M_1$}}
\lbl[]{160,150;{\small $(V,H)$}}
\lbl[]{150,70;{\small $W$}}
\lbl[]{117,160,90;{\small $\longleftarrow$}}
\lbl[]{117,185;{\small $(\partial M_0, h_0)$}}
\lbl[]{205,160,90;{\small $\longleftarrow$}}
\lbl[]{205,185;{\small $(\partial M_1, h_1)$}}
\end{lpic}
\caption{The equivalence relation in $R^{\rm inv}_n$.}
\label{figure_CONC_Rn_EQUIVALENCE}
\end{figure}

Note that in contrast to the concordance theory of positive scalar 
curvature metrics the definition of the $R^{\rm inv}_n$ groups does not 
involve a fixed fundamental group. The reason is that the Surgery 
Theorem \ref{thm_handle_attachment} also allows handle attachment 
of codimension $2$. Thus, any manifold and the manifold obtained by 
killing the fundamental group via codimension $2$ surgeries represent 
the same element in $R^{\rm inv}_n$.

The set $R^{\rm inv}_n$ is an abelian group when addition is defined as
disjoint union and the zero element is given by the equivalence class
of the empty manifold. The groups $R^{\rm inv}_n$ are defined to fit in
the sequence of abelian groups
\[
\ldots \to
R^{\rm inv}_{n+1}
\overset{\partial}{\to} 
\Oinv_n
\overset{i}{\to}
\Ospin_n
\overset{j}{\to}
R^{\rm inv}_n
\to \ldots
\]
where the maps are defined by 
$\partial([M,g]) \definedas [\partial M,g]$, $i([M,g]) \definedas [M]$, 
and $j([M]) \definedas [M,-]$. It is not complicated to see that 
this sequence is exact at $\Oinv_n$ and at $R^{\rm inv}_n$. Exactness at 
$\Ospin_n$ follows from the first point of the Corollary below.

Our main theorem follows \cite[Theorem~5.4]{stolz_xx}.

\begin{thm} \label{thm_exist_class}
Let $M$ be a connected spin manifold of dimension $n \geq 4$.
\begin{enumerate}
\item
$\Rinv(M \rel h)$ is nonempty if and only if $[M,h] \in R^{\rm inv}_n$ 
vanishes.
\item
If $\Rinv(M \rel h)$ is nonempty then $R^{\rm inv}_{n+1}$ acts freely and
transitively on $\conc \Rinv(M \rel h)$.
\end{enumerate}
\end{thm}

For closed manifolds we get the following Corollary as a special case.

\begin{corollary} 
Let $M$ be a closed connected spin manifold of dimension $n\geq 4$.
\begin{enumerate}
\item
$\Rinv(M)$ is nonempty if and only if $[M,-]$ is zero in $R^{\rm inv}_n$. 
\item
If $\Rinv(M)$ is nonempty then $R^{\rm inv}_{n+1}$ acts freely and
transitively on $\conc \Rinv(M)$.
\end{enumerate}
\end{corollary}

\subsection{Proof of the Theorem}

The next Lemma is similar to \cite[Lemma~4.3]{stolz_xx}.

\begin{lemma} \label{lemma_extension}
Let $M$ be a manifold of dimension $n$, and let 
$h \in \Rinv(\partial M)$. Suppose $C$ is obtained from $M$ by 
removing the interior of a compact codimension $0$ submanifold 
$N \subset M$. Assume $h$ can be extended to a metric $H$ on $C$ 
with invertible Dirac operator (and product near the boundary as 
usual). Then $(M,h) \sim (N,H|_{\partial N})$, so they define the 
same element in $R^{\rm inv}_n$.
\end{lemma}

\begin{proof}
For the equivalence of $(M,h)$ and $(N,H|_{\partial N})$ in $R^{\rm inv}_n$
the connecting part $V$ consists of $(C,H)$ with the cylinder 
$(\partial M \times I, h + dt^2)$ attached at $\partial M$, which 
has invertible Dirac operator by assumption. The role of $W$ is 
played by $\widetilde{M \times I}$ which is the product manifold 
$M \times I$ with corners rounded, see 
Figure~\ref{figure_CONC_LEMMA_EXTENSION}. 
\begin{figure}[h]
\begin{lpic}[]{FIGS/CONC_LEMMA_EXTENSION(0.25)}
\lbl[]{60,145;{\small $N$}}
\lbl[]{110,145;{\small $C$}}
\lbl[]{170,145;{\small $\partial M \times I$}}
\lbl[]{235,145;{\small $M$}}
\lbl[]{140,70;{\small $\widetilde{M \times I}$}}
\end{lpic}
\caption{}
\label{figure_CONC_LEMMA_EXTENSION}
\end{figure}
\end{proof}

The following Corollary is immediate.

\begin{cor} \label{cor_extension}
If $h$ extends to a metric with invertible Dirac operator on all of
$M$, then $[M,h] = 0$ in $R^{\rm inv}_n$.
\end{cor}

The following ``Extension Theorem'' is similar to
\cite[Theorem~5.5]{stolz_xx}.

\begin{thm} \label{thm_extension}
Let $M$, $V$ be spin manifolds of dimension $n \geq 4$ with boundary. 
Assume that $\partial M = \partial V$ and that there is a spin bordism 
from $V$ to $M$ for which the boundary bordism is a product 
$\partial M \times I$. If the inclusion 
$M \hookrightarrow W$ is a $1$-equivalence, then a metric 
$g \in \Rinv(V \rel h)$ can be extended to a metric $G$ on $W$ with 
invertible Dirac operator such that $G = h + dt^2$ on the boundary 
bordism.
\end{thm}

\begin{proof}
The bordism $W$ can be built from $V \times I$ by attaching handles of 
codimension $\geq 2$ outside a neighborhood of $\partial V \times I$. 
For closed manifolds this is proved in
\cite[Chapter~8, Proposition~3.1]{kosinski_93}, the argument works also
in our setting. By Theorem~\ref{thm_handle_attach_corners} a metric 
in $\Rinv(V \rel h)$ can be extended over $W$ as required. 
\end{proof}

We prove that every element of $R^{\rm inv}_n$ is represented by the ball 
$B^n$ and a metric on its boundary, this is parallel to 
\cite[Proposition~5.8]{stolz_xx}. We define addition on 
$\conc \Rinv(S^{n-1})$ by taking connected sum of metrics with invertible 
Dirac operator. This makes $\conc \Rinv(S^{n-1})$ into an abelian group. 
Recall that $\conc \Rinv(M)$ denotes the set of concordance classes as 
given in Definition \ref{def_concordance}.

\begin{prop} \label{prop_B^n_repr}
Let $n \geq 5$. For $[M,h] \in R^{\rm inv}_n$ there is a 
$q \in \Rinv(S^{n-1})$ so that $[M,h] = [B^n, q]$. The inclusion of 
elements of the form $[B^n, q]$ into $R^{\rm inv}_n$ induces a group 
isomorphism between $R^{\rm inv}_n$ and $\conc \Rinv(S^{n-1})$.
\end{prop}

\begin{proof}
Let $[M,h] \in R^{\rm inv}_n$. By making surgeries in the interior we may 
assume that $M$ is connected and simply connected. Take an embedding of 
$B^n$ in the interior of $M$ and apply the Extension 
Theorem~\ref{thm_extension} to the $1$-equivalence 
$S^n\hookrightarrow W = M \setminus \operatorname{int} B^n$. This gives
a metric which extends $h$ on $\partial M$ to a metric $G$ on $W$. 
From Lemma~\ref{lemma_extension} the first statement then follows with 
$q$ taken as the restriction of $G$ to $S^{n-1}$.

If $[M_0, h_0] = [M_1,h_1]$ in $R^{\rm inv}_n$, then there is a bordism 
$(V,H)$ from $\partial M_0$ to $\partial M_1$ with $H \in \Rinv(V)$ such 
that $H|_{\partial V} = h_0 \sqcup h_1$. From this it is not complicated to 
use Theorem~\ref{thm_extension} to find a concordance between the 
corresponding $q_0, q_1 \in \Rinv(S^{n-1})$. Further, it is easy to see
that the disjoint union $[M_0 \sqcup M_1, h_0 \sqcup h_1]$ corresponds
to $[B^n \sqcup B^n, q_0 \sqcup q_1]$, which in turn is equivalent to 
the pair consisting of $B^n$ and the connected sum metric $q_0 \# q_1$ 
on $S^{n-1}$.
\end{proof}

We are now ready to prove the first part of 
Theorem~\ref{thm_exist_class}.

\begin{proof}[Proof of Theorem~\ref{thm_exist_class}, (1)] 
From Corollary~\ref{cor_extension} we know that $[M,h] = 0$ if $h$
extends to a metric with invertible Dirac operator on all of $M$,
which is one direction of the claim. 

For the other direction, suppose $[M,h] = 0$. This means that $(M,h)$
is equivalent to the empty manifold. By definition of the
equivalence relation we then know that 
\begin{itemize}
\item 
there is a manifold $V$ with $\partial V = \partial M$, and
a metric $H \in \Rinv(V)$ with $H|_{\partial V} = h$, 
\item 
there is a manifold $W$ with boundary $M \cup_{\partial M} V$,
\item 
all manifolds have compatible spin structures,
\end{itemize}
see the left of Figure~\ref{figure_CONC_THM_EXISTENCE}. By performing 
surgeries in the interior we may change $W$ to be connected and simply 
connected. Then, we introduce corners so that $W$ becomes a bordism 
from $V$ to $M$ which is a product vertical bordism of the boundaries, 
see the right of Figure~\ref{figure_CONC_THM_EXISTENCE}.
\begin{figure}[h]
\begin{lpic}[]{FIGS/CONC_THM_EXISTENCE(0.25)}
\lbl[]{70,135;{\small $M$}}
\lbl[]{135,135;{\small $(V,H)$}}
\lbl[]{60,70;{\small $W$}}
\lbl[]{100,145,90;{\small $\longleftarrow$}}
\lbl[]{100,165;{\small $(\partial M, h)$}}
\lbl[]{230,155;{\small $M$}}
\lbl[]{220,90;{\small $W$}}
\lbl[]{215,20;{\small $(V,H)$}}
\lbl[]{325,110;{\small $(\partial M \times I, h + dt^2)$}}
\end{lpic}
\caption{}
\label{figure_CONC_THM_EXISTENCE}
\end{figure}
Since $M$ is connected the inclusion $M \hookrightarrow W$ is a 
$1$-equivalence, and from the Extension Theorem~\ref{thm_extension} 
we conclude that the metric $H$ extends to a metric on $W$ with 
invertible Dirac operator. In particular this metric, when restricted to
$M$, gives an invertible extension of $h$ to $M$. 
\end{proof}

Next we prove the second part of Theorem~\ref{thm_exist_class}. For
this we follow \cite{stolz_xx} and construct a pairing
\[
i: \conc \Rinv (M \rel h) \times \conc \Rinv (M \rel h) 
\to R^{\rm inv}_{n+1}
\]
with the properties
\begin{itemize}
\item
$i([g_0], [g_1]) + i([g_1],[g_2]) = i([g_0],[g_2])$ for 
$[g_0], [g_1], [g_2] \in \conc \Rinv (M \rel h)$, 
\item
For every $[g_0] \in \conc \Rinv (M \rel h)$ the map 
$i_{[g_0]} : \conc \Rinv (M \rel h) \to R^{\rm inv}_{n+1}$ is a bijection,
where $i_{[g_0]}$ is defined by $i_{[g_0]} ([g]) = i([g_0], [g])$. 
\end{itemize}
Using this pairing we define an action of $R^{\rm inv}_{n+1}$ on 
$\conc \Rinv (M \rel h)$ by $x \cdot [g] = i_{[g]}^{-1} (x)$ for 
$x \in R^{\rm inv}_{n+1}$ and $[g] \in \conc \Rinv (M \rel h)$. From the 
first property of $i$ it follows that this defines an action, and from 
the second property it follows that the action is free and transitive.

As a first step we define the pairing on metrics,
\[
i: \Rinv (M \rel h) \times \Rinv (M \rel h) \to R^{\rm inv}_{n+1}.
\]
Let $\widetilde{M \times I}$ be $M \times I$ with the corners 
rounded. Then $\partial (\widetilde{M \times I}) = 
(-M) \cup \partial M \times I \cup M$. Take 
$g_0, g_1 \in \Rinv (M \rel h)$. By stretching the interval 
$I$ we may assume that the metric $g_0 \cup h + dt^2 \cup g_1$ has 
invertible Dirac operator on the closed manifold 
$\partial (\widetilde{M \times I})$, 
see Section~\ref{subsec_manifolds_with_corners}. We define
\[
i(g_0, g_1) \definedas 
[\widetilde{M \times I}, g_0 \cup h + dt^2 \cup g_1]
\in R^{\rm inv}_{n+1}.
\]

\begin{lemma} \label{lemma_i_conc}
If $g_0$ and $g_1$ are concordant then $i(g_0, g_1) = 0$.
\end{lemma}

\begin{proof}
The fact that $g_0$ and $g_1$ are concordant means that the metrics
extend to a metric with invertible Dirac operator on $M \times
I$. By the discussion in Section~\ref{subsec_manifolds_with_corners}
we get a metric with invertible Dirac operator on 
$\widetilde{M \times I}$ which has $g_0 \cup h + dt^2 \cup g_1$ as 
boundary. From Corollary~\ref{cor_extension} we get that $i(g_0, g_1) 
= [\widetilde{M \times I}, g_0 \cup h + dt^2 \cup g_1] = 0$. 
\end{proof}

In particular, $i(g_0,g_0)=0$. Let $\sigma$ be a metric on a 
two-dimensional hexagonal domain $S$ which is a product near the 
boundary lines and a double product near the corners.

\begin{lemma} \label{lemma_i_add}
For $g_0, g_1, g_2 \in \Rinv (M \rel h)$ we have
\[
i(g_0, g_1) + i(g_1, g_2) = i(g_0, g_2).
\]
\end{lemma}

\begin{proof}
Since $h \in \Rinv(\partial M)$ we have that 
$h + \sigma \in \Rinv(\partial M \times S)$.
Attach $(M \times I, g_i + dt^2)$, $i=0,1,2$, to 
$(\partial M \times S, h + \sigma )$ as in Figure
\ref{figure_CONC_LEMMA_i_ADD}. 
This gives $(V,H)$ in the equivalence
relation for $R^{\rm inv}_{n+1}$.
\begin{figure}[h]
\begin{lpic}[]{FIGS/CONC_LEMMA_i_ADD(0.25)}
\lbl[]{0,110;{\small $(M \times I, g_0 + dt^2)$}}
\lbl[]{290,140;{\small $(M \times I, g_1 + dt^2)$}}
\lbl[]{290,70;{\small $(M \times I, g_2 + dt^2)$}}
\lbl[]{135,20;{\small $(\partial M \times S, h + \sigma)$}}
\end{lpic}
\caption{}
\label{figure_CONC_LEMMA_i_ADD}
\end{figure}
If we glue this manifold with three copies of 
$\widetilde{M \times I}$ we get a closed manifold diffeomorphic to 
$\partial M \times D^2 \cup M \times S^1$, which is the boundary of 
$M \times D^2$ with corner rounded. We set $W$ in the equivalence
relation for $R^{\rm inv}_{n+1}$ to be $M \times D^2$ with corner rounded.
With $(V,H)$ and $W$ chosen like this we conclude that 
$i(g_0, g_1) + i(g_1, g_2) + i(g_2, g_0) = 0$ in $R^{\rm inv}_{n+1}$.
Setting $g_1 = g_0$ we see that $i(g_2,g_0) = - i (g_0,g_2)$, and the
claim of the Lemma follows.
\end{proof}

\begin{proof}[Proof of Theorem~\ref{thm_exist_class}, (2)]
We define the pairing $i$ by 
$i([g_0], [g_1]) \definedas i(g_0, g_1)$. 
From Lemma~\ref{lemma_i_conc} and Lemma~\ref{lemma_i_add} it follows
that $i$ is well-defined and satisfies the addition property. 

We now prove that $i_{[g_0]}$ is injective. Suppose that
$i_{[g_0]} ([g_1]) = i_{[g_0]} ([g_2])$, then 
\[\begin{split}
0 
&= - i_{[g_0]} ([g_1]) + i_{[g_0]} ([g_2]) \\
&= - i( [g_0], [g_1] ) + i( [g_0], [g_2] ) \\
&= i( [g_1], [g_0] ) + i( [g_0], [g_2] ) \\
&= i( [g_1], [g_2] ).
\end{split}\]

If the interval $I$ is long enough the metric 
$g_1 \cup h + dt^2 \cup g_2$ has invertible Dirac operator on 
$(-M) \cup \partial M \times I \cup M$. Since $i( [g_1], [g_2] ) 
= [\widetilde{M \times I}, g_1 \cup h + dt^2 \cup g_2 ] = 0$ it
follows from part (1) of Theorem~\ref{thm_exist_class} that the metric 
$g_1 \cup h + dt^2 \cup g_2$ extends to a metric $G \in 
\Rinv (\widetilde{M \times I} \rel (g_1 \cup h + dt^2 \cup g_2))$, 
again for all sufficiently long intervals $I$.

We use the metric $\sigma$ to reintroduce the corners in 
$\widetilde{M \times I}$, see Figure
\ref{figure_CONC_THM_CLASS_INJ}. After suitable 
stretching of the product structures normal to the attaching boundaries
this will give a metric with invertible Dirac operator, that is a
concordance from $g_1$ to $g_2$. We conclude that $[g_1] = [g_2]$, and
$i_{[g_0]}$ is injective. 

\begin{figure}[h]
\begin{lpic}[]{FIGS/CONC_THM_CLASS_INJ(0.25)}
\lbl[]{30,110;{\small $(\widetilde{M \times I}, G)$}}
\lbl[]{180,205;{\small $(M \times I, g_1 + dt^2)$}}
\lbl[]{180,3;{\small $(M \times I, g_2 + dt^2)$}}
\lbl[]{270,110;{\small $(\partial M \times S, h + \sigma)$}}
\end{lpic}
\caption{}
\label{figure_CONC_THM_CLASS_INJ}
\end{figure}

Next we prove surjectivity of $i_{[g_0]}$. By Proposition
\ref{prop_B^n_repr} we know that any element of $R^{\rm inv}_{n+1}$ can
be represented as $[B^{n+1},q]$ for some $q \in \Rinv(S^n)$. We must
find $g_1 \in \Rinv(M \rel h)$ such that 
$i_{[g_0]}([g_1]) = i([g_0], [g_1]) = [B^{n+1},q]$.

Remove an open ball from the interior of $M \times I$ and denote the
remaining manifold by $C$. We then have $M \times I 
= B^{n+1} \cup_{S^n} C$. Removing the open ball does not change 
$M \times \{ 1 \} \hookrightarrow M \times I$ being a
$1$-equivalence, so also $M \times \{ 1 \} \hookrightarrow C$ is a
$1$-equivalence. By the Extension Theorem~\ref{thm_extension} we can
extend the metric $g_0 \cup q$ on $M \times \{ 0 \} \cup S^n$ to a
metric $G$ with invertible Dirac operator on $C$. 
Set $g_1 = G|_{M \times \{1\}}$, then 
\[
i([g_0], [g_1]) 
= [\widetilde{M \times I}, g_0 \cup h + dt^2 \cup g_1 ]
= [B^{n+1},q]
\]
by Lemma~\ref{lemma_extension}.
\end{proof}

\subsection{The $R^{\rm inv}_n$ groups and the index}

Using the index of the Dirac operator we can conclude that the group 
$R^{\rm inv}_n$ is non-trivial in certain dimensions. Following Bunke 
\cite{bunke_95} and Stolz \cite{stolz_xx} we sketch the definition 
of the index map 
\[
\theta: R^{\rm inv}_n \to KO_n.
\]

For $[M,h] \in R^{\rm inv}_n$ we extend the metric $h$ to a metric $g$ 
on all of $M$. We view the Dirac operator $D^g$ as a $Cl_n$-linear 
operator on $L^2(\Sigma M_{\infty})$. Let $\chi: \mR \to [-1,1]$ be an 
increasing, odd, smooth function which is constant $\pm 1$ outside a 
bounded interval the size of which is related to the spectral gap of 
$D^h$ on $\partial M$. The pair $(L^2(\Sigma M_{\infty}), \chi(D) )$ 
is then a Kasparov module representing 
$\theta([M,h]) \in KK(\mR,Cl_n) = KO_n$. 
For details, see Section~9 of \cite{stolz_xx}. From Theorem~1.2 of 
\cite{bunke_95} it follows that $\theta: R^{\rm inv}_n \to KO_n$ is 
well-defined.

For a compact manifold $M$ without boundary the index map coincides
with the ordinary index, $\theta([M,-]) = \alpha(M)$. Since $\alpha$ is 
surjective we conclude that $\theta$ is also surjective. Further, if 
$KO_n$ is non-trivial then $R^{\rm inv}_n$ is also non-trivial.

From this observation we get a result on existence of metrics with
harmonic spinors, see Hitchin \cite{hitchin_74} and B\"ar \cite{baer_96} 
for the case of closed manifolds. 

\begin{thm}
Let $M$ be a spin manifold with boundary, $\dim M = n$ and 
$h \in \Rinv(\partial M)$. Assume $n$ is such that $R^{\rm inv}_{n+1}$ 
is non-trivial, for example $n \equiv 0,1,3,7 \mod 8$. Then
there is a metric on $M$ which extends $h$ and has non-trivial 
harmonic $L^2$-spinors.
\end{thm}

\begin{proof}
If $\Rinv(M \rel h)$ is empty then all metrics in 
${\mathcal R}(M \rel h)$ have non-trivial harmonic spinors. If 
$\Rinv(M \rel h)$ is non-empty it must have several components 
by Theorem~5.3~(2), so $\Rinv(M \rel h) \neq {\mathcal R}(M \rel h)$. 
\end{proof}

Inspired by a similar conjecture for the case of positive scalar 
curvature metrics, \cite[Conjecture~5.7]{rosenberg_stolz_01}, we 
make the following conjecture.

\begin{conjecture} \label{conj_thetainv}
The index map $\theta: R^{\rm inv}_n \to KO_n$ is injective.
\end{conjecture}


Injectivity of the index map means that $h \in \Rinv(\partial M)$ 
extends to a metric in $\Rinv(\partial M)$ if and only if the index
$\theta([M,h])$ vanishes.

\section{Genericity of metrics with invertible Dirac operator}
\label{section_genericity}

From the surgery theorem for the Dirac operator on closed manifolds, 
Theorem~\ref{am_da_hu_09}, it follows that generic metrics on a closed 
manifold have the minimal dimension allowed by the index theorem, 
see \cite[Theorem~1.1]{ammann_dahl_humbert_09}. In particular, if 
the index vanishes then a generic metric has invertible Dirac 
operator. The proof uses the fact that if there is one minimal metric 
then generic metrics are minimal, surgery can then be used to produce 
one such metric on a given manifold.

Here the term generic means that the subset of minimal metrics is open
in the $C^1$-topology and dense in the $C^\infty$-topology on the set 
of all Riemannian metrics.

Our goal is to obtain a similar statement for manifolds with boundary.
We begin by proving that if there is one metric with invertible Dirac 
operator then generic metrics with the same boundary have invertible
Dirac operator. 

\begin{prop} \label{open_dense} 
Assume that $\Rinv(M \rel h)$ is nonempty. Then $\Rinv(M \rel h)$ is 
open with respect to the $C^1$-topology and dense with respect to the 
$C^\infty$-topology in ${\mathcal R}(M \rel h)$. 
\end{prop}

To prove Proposition~\ref{open_dense} we need the following lemma.

\begin{lem} \label{equivalence_metrics} 
Let $g, g' \in\mathcal{R}(M_{\infty})$ with the boundary metrics
$\partial g = \partial g' = h$ on $\partial M$. Then the maps 
$g' \mapsto \Vert \beta_{g'}^g\phi\Vert_{L^2(g')}^2$ 
($g' \mapsto \Vert \beta_{g'}^g\phi\Vert_{H^1(g')}^2$) are 
uniformly continuous in $\phi \in \Sigma^g M$ with respect to the 
$C^0$-topology ($C^1$-topology) on $\mathcal{R}(M \rel h)$.
\end{lem}

\begin{proof} 
We start with the case of a closed manifold. In local coordinates 
one sees immediately that the volume element depends continuously 
on $g$ in the $C^0$-topology and that the Christoffel symbols depend 
continuously on $g$ in the $C^1$-topology. Since $\beta_{g'}^g$ is 
fiberwise an isometry, the $L^2$-norm ($H^1$-norm) on a single chart 
depends continuously on the $C^0$-topology ($C^1$-topology) on 
$\mathcal{R}(M)$. Hence, the statement is true for closed manifolds. 

The Lemma in general is proven by decomposing $M_\infty$ into 
$M \cup \partial M\times (0,\infty)$. Since $M$ and $\partial M$ 
are compact and the metrics are constant in the 
$(0,\infty)$-direction the lemma follows. 
\end{proof}

From that lemma we get immediately the following corollary.

\begin{cor}
Let $g, g' \in\mathcal{R}(M_{\infty})$ with the boundary metrics
$\partial g = \partial g' = h$ on $\partial M$. 
Then, the norms $\Vert D^{g'}(\beta_{g'}^g .)\Vert_{L^2(g')}$ and 
$\Vert {}^{g\mkern-4mu} D^{g'}.\Vert_{L^2(g)}$ are equivalent.
In particular, 
$D^{g'}: L^2(\Sigma^{g'} M_\infty)\to L^2(\Sigma^{g'} M_\infty)$ is 
invertible if and only if ${}^{g\mkern-4mu} D^{g'}: 
L^2(\Sigma^{g} M_\infty)\to L^2(\Sigma^{g} M_\infty)$ is.
\end{cor}

\begin{proof}[Proof of Proposition~\ref{open_dense}]
Metrics in $\mathcal{R}(M \rel h)$ are the same on the cylindrical
end, so the essential spectrum is also the same for such
metrics. Since we assume $\Rinv(M \rel h)$ to be nonempty, the
essential spectrum for each metric in $\mathcal{R}(M \rel h)$ is 
$(-\infty, -\Lambda] \cup [\Lambda, \infty)$ where $\Lambda > 0$ 
is the absolute value of the lowest eigenvalue of $D^h$, see 
Proposition~\ref{propbdrygap}. This means that on 
$(-\Lambda, \Lambda)$ the spectrum of any metric in 
$\mathcal{R}(M \rel h)$ is discrete and the dimension of the kernel 
is finite, which allows to carry over the proof from the case of 
closed manifolds, see \cite[Proposition~3.1]{maier_97}.

Due to the corollary above it is enough to examine invertibility of 
the operator ${}^{\overline{g}\mkern-4mu} D^{g}$ for a fixed background 
metric $\overline{g} \in \mathcal{R}(M \rel h)$.

First, one shows that the map $g \mapsto {}^{\overline{g}\mkern-4mu} D^{g}$ 
from $\mathcal{R}(M \rel h)$ to 
$\mathcal{B}(H^1(\overline{g}), L^2(\overline{g}))$ is
continuous in the $C^1$-topology on $\mathcal{R}(M \rel h)$. Here
$\mathcal{B}(H^1(\overline{g}), L^2(\overline{g}))$ denotes the space
of bounded linear operators from $H^1(\overline{g})$ to 
$L^2(\overline{g})$. That ${}^{\overline{g}\mkern-4mu} D^{g} \in 
\mathcal{B}(H^1(\overline{g}), L^2(\overline{g}))$ follows immediately
from the estimate 
\[ \begin{split}
\Vert {}^{\overline{g}\mkern-4mu} D^{g} \phi \Vert_{L^2(\overline{g})} 
&\leq
a\Vert D^g (\beta_{\overline{g}}^g\phi)\Vert_{L^2(g)} 
\leq 
an\Vert \nabla^g (\beta_{\overline{g}}^g\phi) \Vert_{L^2(g)} \\
&\leq an \Vert \beta_{\overline{g}}^g\phi \Vert_{H^1(g)}
\leq 
abn \Vert \phi \Vert_{H^1(\overline{g})}
\end{split} \]
where $a,b$ are constants coming from the equivalence of the norms with 
respect to different metrics, see Lemma~\ref{equivalence_metrics}.

Moreover, if $g\in \Rinv(M \rel h)$ there is a neighbourhood of 
${}^{\overline{g}\mkern-4mu} D^{g}$ with respect to the norm topology on 
$\mathcal{B}(H^1(\overline{g}), L^2(\overline{g}))$ such that all 
operators in this neighbourhood are also invertible. This is deduced 
from the following estimate. If $\epsilon$ is small enough and 
$A \in \mathcal{B}(H^1(\overline{g}), L^2(\overline{g}))$ lies in the 
$\epsilon$-neighbourhood of ${}^{\overline{g}\mkern-4mu} D^{g}$, we have 
\[ \begin{split}
\Vert A\phi\Vert_{L^2(\overline{g})}
&\geq 
\Vert {}^{\overline{g}\mkern-4mu} D^{g}\phi 
- ({}^{\overline{g}\mkern-4mu} D^{g}-A)\phi\Vert_{L^2(\overline{g})}\\
&\geq
\Vert {}^{\overline{g}\mkern-4mu} D^{g}\phi \Vert_{L^2(\overline{g})}
- \Vert ({}^{\overline{g}\mkern-4mu} D^{g}-A)\phi \Vert_{L^2(\overline{g})}\\ 
&\geq
\Vert {}^{\overline{g}\mkern-4mu} D^{g} \phi\Vert_{L^2(\overline{g})} 
- \Vert {}^{\overline{g}\mkern-4mu} D^{g}-A\Vert \Vert
\phi\Vert_{H^1(\overline{g})}\\ 
&\geq 
\Vert {}^{\overline{g}\mkern-4mu} D^{g} \phi\Vert_{L^2(\overline{g})} 
-\epsilon b (\Vert \phi \Vert_{L^2(\overline{g})} 
+\Vert D^{\overline{g}} \phi\Vert_{L^2(\overline{g})})\\ 
&\geq 
\Vert {}^{\overline{g}\mkern-4mu} D^{g} \phi \Vert_{L^2(\overline{g})} 
- \epsilon b (\Vert \phi\Vert_{L^2(\overline{g})} 
+ a \Vert {}^{\overline{g}\mkern-4mu} D^{g} \phi\Vert_{L^2(\overline{g})}) \\ 
&\geq 
(1 - ab\epsilon)\Vert {}^{\overline{g}\mkern-4mu} D^{g} \phi \Vert_{L^2(\overline{g})} 
-\epsilon b \Vert \phi\Vert_{L^2(\overline{g})} \\ 
&\geq 
(C(1-ab\epsilon)-b\epsilon) \Vert \phi\Vert_{L^2(\overline{g})},
\end{split} \]
where $b$ is the constant describing the equivalence of the norms 
$\Vert \cdot \Vert_{H^1(\overline{g})}$ and 
$\Vert \cdot \Vert_{L^2(\overline{g})} 
+ \Vert D^g (\cdot) \Vert_{L^2(\overline{g})}$, see for example 
\cite[Prop. 2.7]{mueller_96}, $a$ is the constant describing the 
equivalence of the $L^2$-norms of $D^g (\beta_{\overline{g}}^g\phi)$ and 
${}^{\overline{g}\mkern-4mu} D^{g}\phi$, and $C>0$ is the infimum 
of the $L^2$-spectrum of $D^g$. Together with the $C^1$-continuity of 
$g \mapsto D^g$ this shows that $\Rinv(M \rel h)$ is open in 
$\mathcal{R}(M \rel h)$ with respect to the $C^1$-topology. 

Now let $g_0 \in \Rinv(M \rel h)$ and 
$g_1 \in \mathcal{R}(M \rel h)$. Then $g_t = (1-t)g_0 + tg_1$, 
$t\in [0,1]$, is a path in $\mathcal{R}(M \rel h)$. The corresponding 
family of Dirac operators $D_t \definedas D^{g_t}$ is analytic in $t$, 
see \cite[Section~11]{maier_97}. 

We follow the proof of \cite[Proposition~11.4]{maier_97} and show that 
the set $T \definedas \{t \in (0,1) \mid \dim\ker D_t > 0 \}$ is 
discrete from which it follows that $\Rinv(M \rel h)$ is $C^\infty$-dense 
in $\mathcal{R}(M \rel h)$. Assume that $s \not\in T$, that is $D_s$ is 
invertible. Then $t \not\in T$ for all $t$ in a neighbourhood of $s$, 
so $T$ is closed. Let now $s \in \partial T \cap (0,1)$. We have the 
orthogonal splittings $H^1 = K \oplus H$ and $L^2 = C \oplus D$ where 
$K = \ker D_s$ and $C = D_s(H^1)$. Recall that $K$ is finite-dimensional. 
This induces the decomposition 
\[
D_t=
\begin{pmatrix} 
a_t & b_t \\ 
c_t & d_t
\end{pmatrix}.
\] 
Note that $d_t: H \to D$ at is invertible at $s=t$, and thus also for
$t$ near $s$. If $(x_1,x_2) \in \ker D_t$ for $t$ with invertible
$d_t$, then $a_t(x_1)=-b_t(x_2)$ and $c_t(x_1)=-d_t(x_2)$. Thus,
$x_2=-d_t^{-1}\circ c_t(x_1)$ and $R_t(x_1) \definedas 
(b_t \circ d_t^{-1} \circ c_t - a_t)(x_1) = 0$ where $R_t:K \to C$. 
Hence we always have that $\dim \ker D_s \geq \dim \ker D_t$ and in
particular $\dim \ker D_s = \dim \ker D_t$ if and only if 
$R_t \equiv 0$. Assume that there is a half-closed interval 
$I \subset T$ starting or ending at $s$. For $t\in I$ we have 
$\ker D_t \neq \{0\}$ and thus $\det R_t = 0$. But $R_t$ depends 
analytically on $t$ which then implies that $R_t=0$ in the entire 
neighbourhood of $s$ where $d_t$ is invertible. This contradicts 
$s \in \partial T$ since it implies that there is a sequence 
$t_i\to s$ with $\dim \ker D_{t_i}=0$ and hence $\det R_{t_i} \neq 0$.
\end{proof}

From Theorem~\ref{thm_exist_class} we conclude the following.

\begin{thm} \label{thm_generic_metrics}
Let $M$ be an $n$-dimensional spin manifold with boundary, and 
let $h \in \Rinv(\partial M)$. Then $\Rinv(M \rel h)$ is 
generic in ${\mathcal R}(M \rel h)$ if and only if $[M,h] = 0$ 
in $R^{\rm inv}_n$.
\end{thm}

If Conjecture~\ref{conj_thetainv} holds, then the metrics with 
invertible Dirac operator are generic if and only if the index 
$\theta([M,h])$ vanishes.

\bibliographystyle{amsplain}
\bibliography{invhandle} 

\end{document}